\documentclass{article}
\usepackage[T1]{fontenc}
\usepackage{microtype}
\usepackage[a4paper]{geometry}
\usepackage{amsfonts}
\usepackage{amsthm}
\usepackage{latexsym}
\usepackage{dsfont}
\usepackage{amssymb}
\usepackage{amsmath}
\usepackage{mathtools}
\usepackage{xfrac}
\usepackage{graphicx}
\usepackage{wrapfig}
\usepackage{caption}
\usepackage{tikz}
\usepackage[hidelinks]{hyperref}
\usepackage{enumerate}

\numberwithin{equation}{section}

\theoremstyle{definition}

\swapnumbers
\theoremstyle{plain}
\newtheorem{Thm}{Theorem}[section]  
\newtheorem{Lemma}[Thm]{Lemma}
\newtheorem{Prop}[Thm]{Proposition}

\theoremstyle{definition}
\newtheorem{Def}[Thm]{Definition}
\newtheorem{Exa}[Thm]{Example}

\theoremstyle{remark}
\newtheorem{Rem}[Thm]{Remark}

\newcommand{\N}{\mathbb{N}}

\newcommand{\R}{\mathbb{R}}

\newcommand{\e}{\mathrm{e}}

\newcommand{\B}{\mathcal{B}}
\newcommand{\D}{\mathcal{D}}

\newcommand{\F}{\mathcal{F}}
\newcommand{\G}{\mathcal{G}}
\newcommand{\cC}{\mathcal{C}}

\newcommand{\I}{\mathcal{I}}

\newcommand{\cS}{\mathcal{S}}
\newcommand{\cQ}{\mathcal{Q}}
\newcommand{\cO}{\mathcal{O}}

\newcommand{\Prob}{\mathbf{P}}
\DeclareMathOperator{\Var}{\mathbb{V}}
\newcommand{\E}{\mathbf{E}}

\newcommand{\1}{\mathds{1}}
\newcommand{\comp}{\mathsf{c}}

\newcommand{\conv}[2]{\ {\underset{#1\to #2}{\longrightarrow}}\ }

\newcommand{\pconv}[2]{\ {\underset{#1\to #2}{\overset{\mathbb{P}}{\longrightarrow}}}\ }
\newcommand{\Lconv}[2]{\ {\underset{#1\to #2}{\overset{L^1}{\longrightarrow}}}\ }

\newcommand{\convn}{\ {\underset{n\to \infty}{\longrightarrow}}\ }
\newcommand{\asympeq}[2]{\ {\underset{#1\to #2}{\sim}}\ }

\newcommand{\coloneqq}{\vcentcolon=}

\newcommand{\eps}{\varepsilon}

\newcommand{\Fbar}{\overline{F}}

\newcommand{\Wbar}{\overline{W}}
\def\cal#1{\mathcal{#1}}

\def\d{\mathrm{d}}
\def\L{L}

\def\Zhat{Z}
\def\Zbar{\overline{Z}}
\def\muhat{F}
\def\cC{\mathcal{C}}
\def\bbP{\mathbb{P}}
\def\bbE{\mathbb{E}}

\def\error#1{o_{T,n}^\bbP(#1)}

\begin{document}
\title{Large deviations for the branching random walk with heavy-tailed associated random walk -- a principle of one big jump}
\author{Jakob Stonner}

\maketitle

\begin{abstract}
We prove a version of Nagaev's theorem for the branching random walk with heavy-tailed associated random walk. 
For a branching random walk on $\mathbb{R}$ we consider the random measure $Z_n = \sum_{|u|=n} \e^{-V_u} \delta_{V_u}$ where $V_u$, $|u|=n$ denote the positions of the particles in the $n$-th generation.
Under the assumption that $\bbE[Z_1(\cdot)]$ is a probability distribution with regularly varying tail, we prove that $Z_n((n\mathbb{E}[X] + t_n, \infty)) = W n \mathbb{P}(X > t_n)(1 + o(1))$ in $L^1$ as $n \to \infty$ where $W$ is a non-zero random variable, $t_n \uparrow \infty$ grows suitably fast, and $X$ has law $\bbE[Z_1(\cdot)]$.
The result is explained probabilistically by a principle of one big jump for the branching random walk.

\smallskip

\\{\bf Subclass:} MSC: 60J80 $\cdot$ 60F10
\end{abstract}

\section{Introduction}
In this article we prove a version of Nagaev's theorem for the branching random walk.
Nagaev's theorem \cite{Nagaev1979} states that, for random walk $(S_n)_{n \in \N}$ with zero mean, unit variance, and regularly varying tail of index $-p < -2$ (i.e., $\Prob(S_1 > \lambda t)/\Prob(S_1 > t) \to \lambda^{-p}$ as $t \to \infty$ for all $\lambda > 0$) we have
\begin{equation}\label{eq:Nagaev-preview}
     \Prob(S_n > n\E[S_1] + t) = n\Prob(S_1 > t)(1 + o(1))
\end{equation}
as $n \to \infty$ uniformly in $t \geq a \sqrt{n \log n}$, for any fixed $a > \sqrt{p-2}$.
This formula is explained probabilistically by the \emph{principle of one big jump}: 
Due to its heavy tail, the `cheapest' way for a random walk to deviate substantially from its mean is for one of its increments to make a `big jump'.
To present our version of this result for the branching random walk, let us first define the model.
The branching random walk is a collection of particles $u \in \I$ which have positions $(V_u)_{u \in \I}$ in $\R$ that are constructed in the following way.
Let $\xi = \sum_{j=1}^N \delta_{X_j}$ be a point process on the real numbers.
We start with a single particle located at $0$, forming the generation $0$.
Now in the $(n+1)$-th step of the construction of the process, each particle $u$ from generation $n$ receives an independent identically distributed (i.i.d.) copy of $\xi$ that we shift by the position $V_u$ of $u$.
Then the $(n+1)$-th generation is formed by particles that are located according to the shifted point processes assigned to the particles of the $n$-th generation.
A precise model definition will be given in the next section.

Let
\begin{equation}\label{eq:Laplace-trafo}
     m(\theta) \coloneqq \bbE\bigg[ \int \e^{-\theta x} \xi(\d x)\bigg] = \bbE\bigg[ \sum_{j=1}^N \e^{-\theta X_j} \bigg], \quad \theta \in \R
\end{equation}
denote the Laplace transform of the intensity measure of $\xi$.
We assume that $m(1) = 1$ and we have $m(\theta) = \infty$ for all $\theta < 1$.
Define
\begin{equation}\label{eq:Zhat-def}
     \Zhat_n \coloneqq \sum_{|u| = n} \e^{-V_u}\delta_{V_u},\quad n \in \N_0.
\end{equation}
Here the sum ranges over all particles in generation $n$.
The random measure $\Zhat_n$ is a natural object to describe the positions of particles in large generations (see e.g.~\cite{Biggins1992}).
Moreover, its intensity measure $\bbE[\Zhat_n(\cdot)]$ is the law of the $n$-th step of a random walk with increment law $F \coloneqq \bbE[\Zhat_1(\cdot)]$, which we call the \emph{associated random walk}.
This random walk is an essential tool for investigation of the model (see e.g.~\cite{Lyons1997,Aidekon2014,Shi2015}).

We assume further that $F$ has a regularly varying tail of index $-p$ for some $p > 2$.
Then our version of \eqref{eq:Nagaev-preview} for the branching random walk reads
\begin{equation*}
     \Zhat_n (n\E[S_1] + t_n,\,\infty) = W n \Prob(S_1 > t_n)(1 + o(1))
\end{equation*}
in $L^1$, for every suitably fast growing sequence $t_n \uparrow \infty$ (see Theorem \ref{thm:brw-large-deviations} below).
Here $W$ is the limit of \emph{Biggins' martingale} $W_n \coloneqq \Zhat_n(\R)$ as $n \to \infty$.
The intuition behind this result is that, just as for the random walk, the `cheapest' way for any individual particle's position in the branching random walk to deviate substantially from its expected position is to allow exactly one of its increments to make a `big jump'.
The main contribution to $\Zhat_n (n\E[S_1] + t_n,\,\infty)$ is then comprised of such particles.

\subsection{Notation and overview}
We write $\N = \{1,2,\ldots\}$ and $\N_0 = \N \cup \{0\}$.
For real numbers $a$ and $b$ we denote $a \wedge b = \min\{a,b\}$ and $a \vee b = \max\{a,b\}$.
For sequences $(a_n)_{n \in \N}$, $(b_n)_{n \in \N}$ of real numbers we write $a_n \in o(b_n)$ or $a_n \ll b_n$ if $a_n/b_n \to 0$ as $n \to \infty$.
We also write $a_n \in \cO(b_n)$ if $\limsup_{n \to \infty} a_n/b_n < \infty$.
Moreover, $a_n \sim b_n$ means that $a_n/b_n \to 1$ as $n \to \infty$.
Analogously, if $f,g:\R \to \R$ are functions (such that $g(t)$ is eventually positive for large $t$), then $f(t) \sim g(t)$, $f(t) \in o(g(t))$ and $f(t) \in \cO(g(t))$ are defined similarly.
We use set notation with $\cO(g(t))$ and $o(g(t))$, so e.g.~$f(t) \in \cO(g_1(t)) \cap \cO(g_2(t))$ means $f(t) \in \cO(g_1(t))$ and $f(t) \in \cO(g_2(t))$.
For random variables $X, Y$ we write $X \preceq Y$ and say that $Y$ stochastically dominates $X$ if we have
\begin{equation*}
     \Prob(X \geq t) \leq \Prob(Y \geq t)
\end{equation*}
for all sufficiently large $t \geq 0$.

The rest of this article is organized as follows.
In Section \ref{sec:results} we define the model, state our result, and provide examples.
The proof is presented in Sections \ref{sec:proof} and \ref{sec:stochdom}.
We give a short summary of the proof in Section \ref{sec:overview}.

\subsection{Notation and overview}
We write $\N = \{1,2,\ldots\}$ and $\N_0 = \N \cup \{0\}$.
For real numbers $a$ and $b$ we denote $a \wedge b = \min\{a,b\}$ and $a \vee b = \max\{a,b\}$.
For sequences $(a_n)_{n \in \N}$, $(b_n)_{n \in \N}$ of real numbers we write $a_n \in o(b_n)$ or $a_n \ll b_n$ if $a_n/b_n \to 0$ as $n \to \infty$.
We also write $a_n \in \cO(b_n)$ if $\limsup_{n \to \infty} a_n/b_n < \infty$.
Moreover, $a_n \sim b_n$ means that $a_n/b_n \to 1$ as $n \to \infty$.
Analogously, if $f,g:\R \to \R$ are functions (such that $g(t)$ is eventually positive for large $t$), then $f(t) \sim g(t)$, $f(t) \in o(g(t))$ and $f(t) \in \cO(g(t))$ are defined similarly.
We use set notation with $\cO(g(t))$ and $o(g(t))$, so e.g.~$f(t) \in \cO(g_1(t)) \cap \cO(g_2(t))$ means $f(t) \in \cO(g_1(t))$ and $f(t) \in \cO(g_2(t))$.
For random variables $X, Y$ we write $X \preceq Y$ and say that $Y$ stochastically dominates $X$ if we have
\begin{equation*}
     \Prob(X \geq t) \leq \Prob(Y \geq t)
\end{equation*}
for all sufficiently large $t \geq 0$.

The rest of this article is organized as follows.
In Section \ref{sec:results} we define the model, state our result, and provide examples.
The proof is presented in Sections \ref{sec:proof} and \ref{sec:stochdom}.
We give a short summary of the proof in Section \ref{sec:overview}.
\section{Model setup and main results}\label{sec:results}
\subsection{Branching Random Walk}\label{subsec:setup}
In the following we define the branching random walk model with increment point process $\xi = \sum_{j=1}^N \delta_{X_j}$.
Here $(X_j)_{j \in \N}$ is a sequence of $\R$-valued random variables and $N$ takes values in $\N_0 \cup \{\infty\}$.
The branching random walk is a collection of particles $u$ that have positions $V_u$ in $\R \cup \{\infty\}$.
We label the particles using the Ulam-Harris tree $\I \coloneqq \bigcup_{n \in \N_0} \N^n$ where $\N^0$ by convention contains only the empty word $\varnothing$.
For $u = (u_1,\ldots,u_n) \in \N^n$ and $v = (v_1,\ldots,v_m) \in \N^m$ with $n,m \in \N_0$ we write
\begin{equation*}
     uv \coloneqq (u_1,\ldots,u_n,v_1,\ldots,v_m) \in \N^{n+m}
\end{equation*}
for the concatenation of $u$ and $v$.
Similarly, for $u = (u_1,\ldots,u_n) \in \N^n$ and $j \in \N$ we write
\begin{equation*}
     uj \coloneqq (u_1,\ldots,u_n,j)
\end{equation*}
and call $u$ the \emph{parent} of $uj$.
Moreover, for $u \in \N^n$ with $n \in \N_0$ we write $|u| = n$ and call it the \emph{generation} of $u$.
There is a partial order on $\I$ given by
\begin{equation*}
u \leq v \quad \Leftrightarrow \quad|u| \leq |v|\text{ and }(v_1,\ldots,v_{|u|}) = u
\end{equation*}
where $v = (v_1,\ldots,v_{|v|})$.
We write $u < v$ if $u \leq v$ and $u \neq v$, $u,v \in \I$.
If $u \leq v$ we call $u$ an \emph{ancestor} of $v$.

The branching random walk is now constructed as follows.
The process starts with a single particle labeled with $\varnothing$ at position $V_\varnothing \coloneqq 0$.
Let $(\xi_u)_{u \in \I}$ be a family of i.i.d.~copies of $\xi$ and denote
\begin{equation*}
\xi_u = \sum_{j=1}^{N(u)}\delta_{X_j(u)},\quad u \in \I.
\end{equation*}
For $u \in \I$ we define recursively the position $V_{uj}$ of the particle with label $uj$, $j \in \N$, by
\begin{equation*}
     V_{uj} \coloneqq \begin{cases}
     V_{u} + X_j(u) & j \leq N(u), \\
     \infty & \text{else}
     \end{cases}
\end{equation*}
with the convention $\infty + x = \infty$ for all $x \in \R \cup \{\infty\}$.
Then the collection $(V_u)_{u \in \I}$ is called branching random walk with increment point process $\xi$.
Further, $(V_u)_{u \in \I}$ is called \emph{CMJ process} 
if $\xi$ is almost surely supported in $[0,\infty)$.
A CMJ process $(V_u)_{u \in \I}$ is naturally interpreted as a population model where $u \in \I$ labels an individuum that is born at time $V_u$ (cf.~\cite{Nerman1981,Jagers1989}).
We occasionally write $V(u)$ instead of $V_u$ if it is notationally more convenient.
Further, we write $\Delta V(uj) \coloneqq X_j(u) = V(uj)-V(u)$ for the displacement of the particle $uj$ relative to its parent $u$.

If $\Phi:\R^\I \to \R$ is a measurable map and $u \in \I$ we define the shifted map
\begin{equation*}
[\Phi]_u: \R^\I \to \R,\quad (x_v)_{v \in \I} \mapsto \Phi\big((x_{uv} - x_u)_{v \in \I} \big).
\end{equation*}
In other words, the operator $[\cdot]_u$ shifts the considered quantity into the sub-tree rooted at $u$ and shifts the positions to make the new root $u$ sit at the origin.
For instance, for $u \in \I$ and $W_1 = \sum_{|u|=1} \e^{-V_u}$ we have
\begin{equation*}
[W_1]_u = \sum_{|v|=1} \e^{-(V_{uv} - V_u)} = \sum_{j=1}^{N(u)}\e^{-X_j(u)}.
\end{equation*}
\subsection{Assumptions, examples and result statement}
Let $\xi$ be a point process on $\R$ with intensity measure $\mu(B) \coloneqq \bbE[\xi(B)]$, $B \subseteq \R$ Borel.
Throughout we assume that $\xi$ is almost surely locally finite.
We define the Laplace transform $m$ by \eqref{eq:Laplace-trafo} and set
\begin{align*}
     m'(\theta) \coloneqq -\int x \e^{-\theta x} \mu(\d x),\quad \theta \in \R,
\end{align*}
if the integral exists.
Note that $m'$ arises from $m$ by formally differentiating (indeed $m'$ equals the derivative of $m$ if both are well-defined).
In the following we consider the assumptions
\begin{gather}
     m(1) = 1, \quad m(\theta) = \infty \text{ for all }\theta < 1 \tag{A1} \label{eq:heavy-tailed} \\
     m'(1) \in (-\infty, 0). \tag{A2} \label{eq:positive-mean}
\end{gather}
\begin{Rem}\label{rem:transform2Malthusian}
     In this article we are interested in values $\theta$ from the boundary of the domain of definition of $m$
     \begin{equation*}
          \D \coloneqq \{\theta \in \R\,:\,m(\theta) < \infty\}.
     \end{equation*}
     Note that if $\D$ is nonempty, then it is an interval.
     Generally we are interested in the left-most $\theta > 0$ such that $m(\theta) < \infty$.
     If for a given point process $\xi = \sum_{j=1}^N \delta_{X_j}$ there exists $\theta > 0$ with $m(\theta) < \infty$ and $m(\theta') = \infty$ for all $\theta' < \theta$
     then we may reduce to \eqref{eq:heavy-tailed} by switching to the point process
     \begin{equation*}
          \tilde{\xi} = \sum_{j=1}^N \delta_{\theta X_j + \log m(\theta)}.
     \end{equation*}
     Therefore, if $\D$ is non-empty, $\inf \D > 0$, and $\inf \D \in \D$ then we may assume \eqref{eq:heavy-tailed} without loss of generality.
\end{Rem}
\vspace{0.2cm}

Let $(V_u)_{u \in \I}$ be a branching random walk with increment point process $\xi$, starting from $0$.
Note that \eqref{eq:heavy-tailed} implies that $\bbE[\xi(\R)] = \infty$, whence the branching process is supercritical.
Let
\begin{equation}\label{eq:Biggins-MG}
     W_n \coloneqq \sum_{|u| = n} \e^{-V_u},\quad n \in \N_0,
\end{equation}
denote Biggins' martingale.
It is well-known (see \cite{Biggins1992, Lyons1997}) that this martingale has a non-zero limit $W$ on survival of the branching process if and only if \eqref{eq:positive-mean} holds and
\begin{equation}\label{eq:Biggins-cond}
     \bbE[W_1 \log_+ W_1] < \infty
\end{equation}
with $\log_+ x \coloneqq 0 \vee \log x$.
We shall assume that there exists $\gamma \in (1,2)$ such that
\begin{equation}\label{eq:Lp-cond}
     \bbE[W_1^\gamma] < \infty \text{ and } m(\gamma) < 1. \tag{A3}
\end{equation}
Under this assumption $W_n$ converges to $W$ in $\L^\gamma$ by \cite[Thm 2.1]{Liu2000}.
Moreover, by assumption \eqref{eq:heavy-tailed} the measure 
\begin{equation*}
     \muhat(\d x) \coloneqq \e^{- x}\mu(\d x)
\end{equation*}
defines a \emph{heavy-tailed} distribution, i.e., it lacks a finite exponential moment.
One possible natural further assumption is therefore that its tail $\Fbar(t) \coloneqq \muhat(t,\,\infty)$ is regularly varying at infinity:
\begin{equation}\label{eq:mu-reg-var}
     \Fbar(t) \asympeq{t}{\infty} t^{-p} \ell(t) \tag{A4}
\end{equation}
for some $p > 0$ and a slowly varying function $\ell$ at $\infty$ (i.e., $\ell(\lambda t)/\ell(t) \to 1$ as $t\to \infty$, for all $\lambda > 0$).
The random walk $(S_n)_{n \in \N_0}$ with increment law $\Prob(S_1 \in (\cdot)) = \muhat$ will be called \emph{associated} with the branching random walk $(V_u)_{u \in \I}$.
Subject to \eqref{eq:positive-mean}, $\muhat$ has finite, positive mean
\begin{equation*}
     c \coloneqq \int x \muhat(\d x) = \E[S_1] = -m'(1) \in (0,\infty).
\end{equation*}
If $p > 2$ (which will be the case throughout the paper), then \eqref{eq:mu-reg-var} and \eqref{eq:heavy-tailed} further imply that $\muhat$ has finite variance
\begin{equation*}
     \sigma^2 \coloneqq \int (x-c)^2 \muhat(\d x) \in (0,\infty).
\end{equation*}
Further, we obtain a stronger result by additionally assuming that there exists $\gamma > 1$ such that
\begin{equation}\label{eq:Zhat-tail-gamma}\tag{A5}
     \bbE[\Zbar_1(t)^\gamma] \in \cO(\Fbar(t)^\gamma) \quad \text{ as }t \to \infty.
\end{equation}
where $\Zbar_1(t) \coloneqq \int_t^\infty \e^{-x} \xi(\d x) = \sum_{|u|=1}\e^{-V_u}\1\{V_u > t\}$, $t \in \R$.
Note that, by Jensen's inequality, if \eqref{eq:Zhat-tail-gamma} is true for some $\gamma > 1$ then it is also true for any $\gamma' \in (1,\gamma)$.
Therefore, if \eqref{eq:Zhat-tail-gamma} and \eqref{eq:Lp-cond} hold we may without loss of generality we may assume that the conditions hold for the same $\gamma \in (1,2)$.
\begin{Rem}
     Assumption \eqref{eq:Zhat-tail-gamma} may be viewed as a strong uniform integrability condition.
     By de la Vall\'ee-Poussin's theorem $(\Zbar_1(t)/\Fbar(t))_{t \in \R}$ is uniformly integrable if and only if there exists a function $G:[0,\infty) \to [0,\infty)$ with $G(t)/t \to \infty$ such that
     \begin{equation*}
          \sup_{t \in \R} \bbE \Big[G\Big( \frac{\Zbar_1(t)}{\Fbar(t)}\Big) \Big] < \infty.
     \end{equation*}
     Assumption \eqref{eq:Zhat-tail-gamma} states that we can choose $G(x) = x^\gamma$ for some $\gamma > 1$, therefore imposing that $(\Zbar_1(t)/\Fbar(t))_{t \in \R}$ is uniformly integrable.
     In particular, assumption \eqref{eq:Zhat-tail-gamma} is never satisfied if $\xi(\R) < \infty$ almost surely since then have $\Zbar_1(t)/\Fbar(t) \to 0$ as $t \to \infty$ almost surely.
\end{Rem}
\begin{Exa}\label{exa:example}     
     (a) 
     Let $\xi$ be a Poisson point process with intensity measure
     \begin{equation*}
          b\,\e^x x^{-(p+1)} \ell(x) \1_{(1,\infty)}(x) \d x
     \end{equation*}
     for some $b > 0$, $p > 2$ and a slowly varying function $\ell$ at infinity.
     Then we have
     \begin{equation*}
          m(\theta) = b \int_1^\infty \e^{-x(\theta - 1)} x^{-(p+1)} \ell(x) \d x
     \end{equation*}
     which is finite if and only if $\theta \geq 1$.
     Moreover, choosing $b^{-1} = \int_1^\infty x^{-(p+1)} \ell(x) \d x$ guarantees \eqref{eq:heavy-tailed}.
     The assumption \eqref{eq:mu-reg-var} follows from Karamata's theorem (see e.g.~\cite[Proposition 1.5.8]{Bingham1987}).
     Since $p > 2$ this implies \eqref{eq:positive-mean}.
     Moreover, Campbell's formula for the variance of integrals with respect to Poisson point processes (e.g.~\cite[Lemma 15.22 (iii)]{Kallenberg2021}) gives that
     \begin{equation*}
          \Var[W_1] = \int_0^\infty (\e^{-x})^2 \bbE[\xi(\d x)] = m(2) < \infty.
     \end{equation*}
     This implies \eqref{eq:Lp-cond}.
     Further, for all $\gamma \in (1,2)$ we have, by Jensen's inequality, Campbell's formula, and subadditivity of $x \mapsto x^{\gamma/2}$,
     \begin{align*}
          \bbE[\Zbar_1(t)^\gamma] &\leq \bbE[ \Zbar_1(t)^2]^{\gamma/2} \\
          &= \big(\Var[\Zbar_1(t)] + \bbE[\Zbar_1(t)]^2 \big)^{\gamma/2} \\
          &= \bigg( \int_t^\infty \e^{-2x} \bbE[\xi(\d x)] + \Fbar(t)^2 \bigg)^{\gamma/2} \\
          &\leq \e^{-\gamma t/2} \Fbar(t)^{\gamma/2} + \Fbar(t)^\gamma.
     \end{align*}
     This shows that \eqref{eq:Zhat-tail-gamma} holds.
     Consequently, $\xi$ is an example of a point process where \eqref{eq:heavy-tailed} through \eqref{eq:Zhat-tail-gamma} are satisfied.
     \vspace{0.2cm}

     \noindent
     (b)
     The following example is inspired by the model that is considered in \cite{Dereich2017} where a Yule process is used instead of a Cox process.
     Let $f$ be a $[0,1)$-valued random variable, $b > 0$, and let $\xi$ be a Cox process with random intensity measure $b\, \e^{f x} \d x$ on $[0,\infty)$.
     Then we have
     \begin{equation*}
          m(\theta) = b\bbE \bigg[\int_0^\infty \e^{-x(\theta - f)} \d x \bigg] = b\bbE[(\theta - f)^{-1}]
     \end{equation*}
     whenever $\theta > f$ almost surely, and $m(\theta) = \infty$ otherwise.
     Now assume that we have
     \begin{equation*}
          \bbP(1 - f \leq t) = t^{p+1}\ell(t),\quad t \leq \eps
     \end{equation*}
     for $\eps \in (0,1)$, $p > 0$ and $\ell$ slowly varying at zero.
     Then one can show that $m(1)$ finite and therefore we can choose $b > 0$ so that $m(1) = 1$.
     As in (a) one shows that \eqref{eq:Lp-cond} is satisfied.
     Moreover, one can show that $t \mapsto \bbE[\e^{-t(1-f)}]$ is regularly varying at $\infty$ with index $-(p+1)$ (details are provided in Appendix \ref{sec:example}).
     Therefore,
     \begin{equation*}
          \Fbar(t) = b \int_t^\infty \bbE[\e^{-x(1-f)}] \d x
     \end{equation*}
     is regularly varying with index $-p$ by Karamata's theorem.
     However, using $\bbE[\Zbar_1(t)\,|\,f] = b \e^{-t(1-f)}/(1-f)$ and Jensen's inequality for conditional expectations we find for every $\gamma > 1$
     \begin{align*}
          \bbE[\Zbar_1(t)^\gamma] &\geq \bbE[ \bbE[\Zhat_1(t)\,|\,f]^\gamma] = b^\gamma \bbE\Big[ \frac{\e^{-\gamma t(1-f)}}{(1-f)^\gamma} \Big] \\
          &\geq b^\gamma \bbE\Big[ \frac{\e^{-\gamma t(1-f)}}{1-f} \Big] = b^{\gamma-1} \Fbar(\gamma t).
     \end{align*}
     This implies that $\bbE[\Zbar_1(t)^\gamma]/\Fbar(t)^\gamma \to \infty$ as $t \to \infty$.
     We conclude that \eqref{eq:heavy-tailed} through \eqref{eq:mu-reg-var} hold but \eqref{eq:Zhat-tail-gamma} is not satisfied.
     \vspace{0.2cm}

     \noindent
     (c)
     Let $N$ be a $\N$-valued random variable such that $\bbE[N^{1/2}] = 1$ and 
     \begin{equation*}
          \bbP(N \geq t) \asympeq{t}{\infty} t^{-1/2} (\log t)^{-(p+1)}.
     \end{equation*}
     Define $\xi = N \delta_{\frac{\log N}2}$.
     Then we have
     \begin{equation*}
          m(\theta) = \bbE[N^{1-\theta/2}],
     \end{equation*}
     so that \eqref{eq:heavy-tailed} is satisfied.
     Moreover, a short calculation shows that
     \begin{equation*}
          \Fbar(t) = \bbE[N^{1/2}\1\{ N > \e^{2t}\}] \asympeq{t}{\infty} c t^{-p}
     \end{equation*}
     for some $c \in (0,\infty)$,
     whence \eqref{eq:mu-reg-var} holds with $\ell = 1$.
     However, we have $W_1 = N^{1/2}$, so \eqref{eq:Biggins-cond} holds but \eqref{eq:Lp-cond} is violated.
     Thus $\xi$ is an example of a point process that satisfies \eqref{eq:heavy-tailed}, \eqref{eq:positive-mean} and \eqref{eq:mu-reg-var} but not \eqref{eq:Lp-cond} and \eqref{eq:Zhat-tail-gamma}.
     This highlights the interest of replacing the assumption \eqref{eq:Lp-cond} by \eqref{eq:Biggins-cond}.
\end{Exa}
\begin{Rem}
     A situation in which \eqref{eq:heavy-tailed} through \eqref{eq:Zhat-tail-gamma} hold may arise in CMJ processes that have no Malthusian parameter, i.e., there is no $\theta > 0$ such that $m(\theta) = 1$.
     If there is a $\theta > 0$ such that $m(\theta) < 1$ and $m(\theta') = \infty$ for all $\theta' < \theta$ (take for instance Example \ref{exa:example} (a) or (b) with smaller $b > 0$) then we can use the transformation in Remark \ref{rem:transform2Malthusian} to get a point process $\tilde{\xi}$ that satisfies \eqref{eq:heavy-tailed}.
     Note that $\tilde{\xi}$ may not be supported in $[0,\infty)$ anymore.
\end{Rem}
\vspace{0.2cm}
Now we state our result. For $n \in \N$, let the random measure $\Zhat_n$ be defined by \eqref{eq:Zhat-def} and let $\Zbar_n(t) \coloneqq \Zhat_n(t,\infty)$, $t \in \R$ denote its tail.
\begin{Thm}\label{thm:brw-large-deviations}
     Assume \eqref{eq:heavy-tailed} through \eqref{eq:mu-reg-var} with $p > 2$ and fix $a > \sqrt{p-2}$ and a sequence $(t_n)_{n \in \N}$ with $t_n \geq a \sigma \sqrt{n \log n}$ for all $n \in \N$.
     Further suppose that one of the following conditions is satisfied:
     \begin{enumerate}[(I)]
          \item 
          There are $\delta \in (0,1-\gamma), q \in (1+\delta,\gamma)$, and a sequence $(r(n))_{n \in \N}$ with $r(n) \in o(n) \cap o(t_n)$ such that 
          \begin{gather}
               \limsup_{n \to \infty} \e^{-(1+\delta)r(n)}\sum_{j=1}^{n-1} \e^{r(j)} < \infty, \quad \text{ and } \label{eq:r-growth} \\
               \limsup_{n \to \infty} \frac{\e^{-r(n)(1 - (1+\delta)/q)}}{n \Fbar(t_n)} < \infty. \label{eq:r2t}
          \end{gather} \label{item:case1}
          \item \eqref{eq:Zhat-tail-gamma} holds. \label{item:case2}
     \end{enumerate}
     Then we have
     \begin{equation}\label{eq:brw-large-deviations}
          \frac{\Zbar_n(nc + t_n)}{n \Prob(S_1 > t_n)} \Lconv{n}{\infty} W
     \end{equation}
     where $W$ is the limit of Biggins' martingale.
     Moreover, in the case \eqref{item:case2} the convergence holds in $L^q$ for any $q < 2 - \gamma^{-1}$.
\end{Thm}
\begin{Rem}
     (a) We can decompose
     \begin{equation*}
          \Zbar_n(t) \eqqcolon \sum_{|u|=n}\e^{-V_u}\1\big\{V_u > t,\,\exists v \leq u:\,\Delta V(v) > h_n \geq \max_{\substack{w \leq u \\ w \neq v}} \Delta V(w) \big\} + \mathrm{rest}_n(t)
     \end{equation*}
     for some fixed sequence $(h_n)_{n \in \N}$.
     The rest term collects contributions to $\Zbar_n(t)$ due to particles which have no or at least two `large jumps' among the displacements of their ancestors, where a displacement is considered to be `large' if it exceeds $h_n$.
     Using the many-to-one lemma (Lemma \ref{lemma:many21} below) and the proof of Nagaev's theorem given in \cite{Denisov2008}) (see Equation (7) and Proposition 8.1 in \cite{Denisov2008}) it follows that for $h_n \coloneqq \sigma \sqrt{n/(a \log n)}$ we have
     \begin{equation*}
          \bbE[\mathrm{rest}_n(t)] \in o(n \Prob(S_1 > t))
     \end{equation*}
     as $n \to \infty$ uniformly in $t \geq nc + a \sigma \sqrt{n \log n}$.
     It follows that only particles which have exactly one `large jump' among the displacements of one of their ancestors contribute to $\Zbar_n(t)$ asymptotically.
     Hence, \eqref{eq:brw-large-deviations} may be explained heuristically by a \emph{principle of one big jump}.
     \vspace{0.2cm}

     \noindent (b)
     One may compare the uniformity of Theorem \ref{thm:brw-large-deviations} to that of Nagaev's theorem (see \eqref{eq:Nagaev-preview}, and Theorem \ref{thm:Nagaev} below). 
     The statement of Theorem \ref{thm:brw-large-deviations} in case \eqref{item:case2} is equivalent with
     \begin{equation*}
          \sup_{t \geq a \sigma \sqrt{n \log n}}\bbE \bigg|\frac{\Zbar_n(nc + t)}{n \Prob(S_1 > t)} - W \bigg| \convn 0.
     \end{equation*}
     It is plausible that under certain conditions \eqref{eq:brw-large-deviations} holds almost surely along any fixed sequence $t_n \geq a \sigma \sqrt{n \log n}$.
     However, if $\xi(\R)$ is finite almost surely (as in Example \ref{exa:example} (c)) then for every $n \in \N$ there exists a large $t > 0$ with $\Zbar_n(nc + t) = 0$, so that 
     \begin{equation*}
          \sup_{t \geq a \sigma \sqrt{n \log n}}\Big|\frac{\Zbar_n(nc + t)}{n \Fbar(t)} - W \Big| \geq |W|.
     \end{equation*}
     This shows that an almost sure convergence with the same kind of uniformity as in Nagaev's theorem is not true generally.
     \vspace{0.2cm}

     \noindent (c)
     Note that, in view of Potter's theorem (e.g.~\cite[Theorem 1.5.6]{Bingham1987}) and $r(n) \in o(n)$, condition \eqref{eq:r2t} in case \eqref{item:case1} implies $\log t_n \ll n$.
     Moreover, the following facts are true.
     \begin{enumerate}[(i)]
          \item
          If $(r(n))_{n \in \N}$ is increasing and $\e^{-\delta r(n)}n \to 0$ as $n \to \infty$ then \eqref{eq:r-growth} holds.
          \item
          If we have
          \begin{equation*}
               \liminf_{n \to \infty} \frac{r(n)}{\log t_n} > \frac{p}{1-\frac{1+\delta}{\gamma}}
          \end{equation*}
          then \eqref{eq:r2t} follows.
          Indeed, there exist $q \in (1,\gamma)$ and $p' > p$ such that for all sufficiently large $n$ we have
          \begin{equation*}
               r(n) \geq \frac{p'}{1 - \frac{1+\delta}{q}} \log t_n.
          \end{equation*}
          Therefore we have
          \begin{equation*}
               \frac{\e^{-r(n)(1 - (1+\delta)/q)}}{n \Fbar(t_n)} \leq \e^{-p' \log t_n + \log \Fbar(t_n)}.
          \end{equation*}
          By Potter's theorem, the exponent on the right-hand side goes to $-\infty$ as $n \to \infty$, whence \eqref{eq:r2t} holds.
          \item
          If $(t_n)_{n \in \N}$ is increasing and satisfies $\log t_n \ll n$ then all conditions of case \eqref{item:case1} hold.
          Indeed, then there exists an increasing sequence $(r(n))_{n \in \N}$ with $\log t_n \ll r(n) \ll t_n \wedge n$.
          In particular, for every $C \in (0,\infty)$ we have $\log t_n \leq C r(n)$ eventually. Using $t_n \geq a \sigma \sqrt{n \log n}$ this yields for every $\delta > 0$
          \begin{equation*}
               \e^{-\delta r(n)} \leq \e^{-\delta C \log(a \sigma \sqrt{n \log n})} = (a \sigma)^{-\delta C} (n \log n)^\frac{-\delta C}{2},
          \end{equation*}
          yielding $\e^{-\delta r(n)}n \to 0$ if $C > 2/\delta$.
          Therefore, from (i) and (ii) we conclude \eqref{eq:r-growth} and \eqref{eq:r2t}.
     \end{enumerate}
     \vspace{0.2cm}

     \noindent (d)
     The following argument shows that in some cases a growth limitation of $(t_n)_{n \in \N}$ is crucial for the conclusion of Theorem \ref{thm:brw-large-deviations} to hold. 
     Suppose \eqref{eq:brw-large-deviations} holds for all $(t_n)_{n \in \N}$ with $t_n \geq a \sigma \sqrt{n \log n}$ (as in case \ref{item:case2}).
     Fix $q > p$ and consider the sum
     \begin{equation*}
          Y_n \coloneqq \sum_{|u| = n} \e^{-V_u} (V_u)_+^q.
     \end{equation*}
     We find for all $t \geq 0$
     \begin{equation*}
          Y_n \geq \sum_{|u| = n} \e^{-V_u} (V_u)_+^q \1\{V_u > t\} \geq t^q \Zbar_n(t) \geq t^q \Zbar_n(nc + t).
     \end{equation*}
     Now let $(a_n)_{n \in \N}$ be an arbitrary sequence of positive numbers.
     We find
     \begin{equation}\label{eq:arbitrary-growth}
          \bbP(Y_n \geq a_n) \geq \bbP(t_n^q \Zbar_n(nc + t_n) \geq a_n) = \bbP\Big( \frac{\Zbar_n(nc + t_n)}{n \Fbar(t_n)} \geq \frac{a_n}{n \Fbar(t_n) t_n^q} \Big) \convn 1
     \end{equation}
     if we choose $t_n$ appropriately.
     Suppose that there are infinitely many $n \in \N$ such that $\bbP(Y_n < \infty) = 1$.
     Then there exist sequences $(n_j)_{j \in \N}$ and $(a_{n_j})_{j \in \N}$ such that $n_j \uparrow \infty$ and $\bbP(Y_{n_j} \leq a_{n_j}) \leq 1/2$ for all $j \in \N$.
     If we now extend $(a_{n_j})_{j \in \N}$ to a sequence $(a_n)_{n \in \N}$ by filling in the missing values arbitrarily, then we obtain
     \begin{equation*}
          \liminf_{n \to \infty} \bbP(Y_n \geq a_n) \leq \frac12,
     \end{equation*}
     contradicting \eqref{eq:arbitrary-growth}.
     Therefore, we have $Y_n = \infty$ with positive probability for all but finitely many $n \in \N$.
     Under the assumption $\xi(\R) < \infty$ almost surely, for instance, this is not possible since every generation is finite almost surely.
     Therefore, there are cases in which \eqref{eq:brw-large-deviations} is not true for arbitrarily fast growing $t_n \uparrow \infty$.
\end{Rem}
\section{Proof of Theorem \ref{thm:brw-large-deviations}}\label{sec:proof}
\subsection{Preliminaries}\label{subsec:preliminaries}
In the following we state well-known results that will be used in the proof of Theorem \ref{thm:brw-large-deviations}.
\subsubsection*{The associated random walk and the coming generation}
\begin{Lemma}[{Many-to-one formula,\cite{Biggins1997a,Shi2015}}]\label{lemma:many21}
     Let $(V_u)_{u \in \I}$ be a branching random walk with $m(1) = \bbE[\sum_{|u|=1} \e^{-V_u}] = 1$.
     Let $(S_n)_{n \in \N}$ be a random walk with increment distribution 
     \begin{equation*}
          \Prob(S_1 \in \d x) = \muhat(\d x) = \e^{- x}\mu(\d x)
     \end{equation*}
     starting at $0$.
     Then for every $n \in \N$ and bounded, measurable $f:\R^{n} \to \R$ we have
     \begin{equation}\label{eq:many21}
          \bbE \bigg[ \sum_{|u|=n} \e^{- V_u} f(V_{u_1},\ldots,V_{u_n}) \bigg] = \E[f(S_1,\ldots,S_n)],
     \end{equation}
     where we write $u = (u_1,\ldots,u_n)$ and $\E[\cdot]$ denotes the expectation with respect to $\Prob$.
     In particular, for every Borel subset $B \subseteq \R$ we have
     \begin{equation*}
          \bbE[\Zhat_n(B)] = \muhat^{\ast n}(B) = \Prob(S_n \in B).
     \end{equation*}
\end{Lemma}
Following \cite{Shi2015}, we call $(S_n)_{n \in \N}$ the \emph{associated random walk} of the branching random walk $(V_u)_{u \in \I}$.
Instead of summing over all the $n$-th generation particles in \eqref{eq:many21} we are interested in summing over the random set
\begin{equation}\label{eq:coming-gen}
     \cC_t \coloneqq \{u \in \I\,|\,V_u > t,\,\forall v < u:\, V_v \leq t\}
\end{equation}
of particles whose position exceeds $t \geq 0$ for the first time along its ancestral line.
In the context of CMJ processes $\cC_t$ known as the \emph{coming generation} at time $t$ (see \cite{Jagers1989}).
If we define the stopping time
\begin{equation}\label{eq:hitting-time}
     \tau_t \coloneqq \inf\{n \in \N\,:\,S_n > t\},\quad t \geq 0
\end{equation}
then it is straight-forward to show that \eqref{eq:many21} implies
\begin{equation}\label{eq:many214lines}
     \bbE\bigg[ \sum_{u \in \cC_t} \e^{-V_u} f(|u|, V_u)\bigg] = \E[f(\tau_t, S_{\tau_t})\1\{\tau_t < \infty\}]
\end{equation}
for every bounded, measurable $f:\N_0 \times \R \to \R$ (see the proof of \cite[Theorem 10]{Biggins2005}).
Note that if $\E[S_1] > 0$ (which is the case under assumption \eqref{eq:positive-mean}) we have $\tau_t < \infty$ almost surely.

Corresponding to $(\cC_t)_{t \geq 0}$ there exists a filtration $(\F_{\cC_t})_{t \geq 0}$ such that, for all $t,s \geq 0$ and $u \in \I$, $\1\{u \in \cC_t\}$ is $\F_{\cC_s}$-measurable if and only if $t \leq s$, see \cite{Jagers1989,Kyprianou2000}.
\begin{Thm}[{\cite[Theorem 9]{Kyprianou2000}}]\label{thm:Nermans-MG}
     Under the assumptions \eqref{eq:heavy-tailed}, \eqref{eq:positive-mean} and \eqref{eq:Biggins-cond},
     the process
     \begin{equation*}
          Y_t \coloneqq \sum_{u \in \cC_t} \e^{-V_u},\quad t \geq 0
     \end{equation*}
     is a unit mean martingale that converges almost surely and in $\L^1$ to the limit $W$ of Biggins' martingale \eqref{eq:Biggins-MG}.
\end{Thm}
In the context of CMJ processes the process $(Y_t)_{t \geq 0}$ is known as \emph{Nerman's martingale} (c.f.~\cite{Nerman1981}), which is also how we will refer to it.

\subsubsection*{Random walks with regularly varying tail}
Let $(S_n)_{n \in \N}$ be a centered random walk with regularly varying tail of index $-p \leq 0$,
i.e.,
\begin{equation}\label{eq:reg-var}
     \lim_{t \to \infty}\frac{\Prob(S_1 > \lambda t)}{\Prob(S_1 > t)} = \lambda^{-p}
\end{equation}
for all $\lambda > 0$ (see \cite{Bingham1987}).
In particular, the distribution of $S_1$ is \emph{sub-exponential}, that is, for all \emph{fixed} $n \in \N$ we have, as $t \to \infty$,
\begin{equation}\label{eq:subexponential}
     \Prob(S_n > t) = n \Prob(S_1 > t)(1 + o(1)),
\end{equation}
see \cite[Appendix 4]{Bingham1987}.
Nagaev's theorem is a version of \eqref{eq:subexponential} where $n$ and $t$ approach $\infty$ simultaneously.
\begin{Thm}[{\cite[Theorem 1.9]{Nagaev1979}}]\label{thm:Nagaev}
     Let $(S_n)_{n \in \N}$ be a random walk with regularly varying tail of index $p > 2$, mean zero and unit variance.
     Fix $a > \sqrt{p-2}$.
     Then we have
     \begin{equation}\label{eq:Nagaev}
          \Prob(S_n > t) = n \Prob(S_1 > t)(1 + o(1))
     \end{equation}
     as $n \to \infty$ uniformly in $t \geq a \sqrt{n \log n}$, i.e.,
     \begin{equation*}
          \sup_{t \geq a \sqrt{n \log n}} \Big| \frac{\Prob(S_n > t)}{n \Prob(S_1 > t)} - 1 \Big| \convn 0.
     \end{equation*}
\end{Thm}
See also \cite{Denisov2008} for a modern (and more general) proof of this result.
Note that regularly varying functions are in particular \emph{long-tailed}, i.e., $S_1$ with regularly varying tail satisfies
\begin{equation}\label{eq:long-tailed}
     \Prob(S_1 > t + x) \asympeq{t}{\infty} \Prob(S_1 > t)
\end{equation}
for all $x \in \R$.
Moreover, if $(S_n)_{n \in \N}$ is a random walk with regularly varying tail of index $-p < -2$ that has mean $c = \E[S_1]$ and variance $\mathbf{V}[S_1] = \sigma^2$, then Theorem \ref{thm:Nagaev} combined with \eqref{eq:long-tailed} implies that we have
\begin{equation*}
	\Prob(S_n > nc + t) = n \Prob(S_1 > t)(1 + o(1))
\end{equation*}
as $n \to \infty$ uniformly in $t \geq a \sigma \sqrt{n \log n}$.
\subsection{Overview}\label{sec:overview}
Our proof of Theorem \ref{thm:brw-large-deviations} has three parts.
First (Section \ref{subsec:proof-part1}) we group the $n$-th generation particles $u \in \I$ with $V_u > nc + t_n$ according to their ancestor whose position first exceeds $r(n) < nc + t_n$.
This gives the decomposition
\begin{equation}\label{eq:Zhat-decomp}
     \Zbar_n(nc+t_n) = \sum_{\substack{u \in \cC_{r(n)} \\ |u| \leq n}} \e^{-V_u} [\Zbar_{n-|u|}(nc+t_n-V_u)]_u
\end{equation}
where $\cC_{r(n)}$ is defined in \eqref{eq:coming-gen} and we have
\begin{equation*}
     [\Zbar_n(t)]_u = \sum_{|v| = n} \e^{-(V_{uv} - V_u)}\1\{V_{uv} - V_u > t\},\quad u \in \I,n \in \N,t \in \R.
\end{equation*}
This decomposition is possible since
every particle $u$ from generation $n$ with $V_u > nc + t$ must itself be in $\cC_{r(n)}$ or have an ancestor in $\cC_{r(n)}$.
Then we use the many-to-one lemma to get rid of summands in \eqref{eq:Zhat-decomp} that correspond to particles with large generations or large overshoots (see Lemma \ref{lemma:approx}). 

In the second part (Section \ref{subsec:proof-part2}, Lemma \ref{lemma:Zhat-wlln}) we use a weak law of large numbers to approximate the right-hand side of \eqref{eq:Zhat-decomp} by
\begin{equation*}
     \sum_{\substack{u \in \cC_{r(n)} \\ |u| \leq n}} \e^{-V_u} \bbE[[\Zbar_{n-|u|}(nc+t_n-V_u)]_u\,|\,V_u] = \sum_{\substack{u \in \cC_{r(n)}\\ |u| \leq n}} \e^{-V_u} \Fbar_{n - |u|}\big(t_n - (V_u-|u|c) \big).
\end{equation*}
where $\Fbar_n(t) \coloneqq \Prob(S_n - nc > t)$, $t \in \R$ is the tail of the distribution of $S_n - nc$.
Finally, in the third part (Section \ref{subsec:proof-part3}) we use Theorem \ref{thm:Nagaev} to approximate $\Fbar_n(t) \approx n \Fbar(t)$ for large $n$ and $t$, so that we obtain
\begin{equation*}
     \Zbar_n(nc+t_n) \approx \sum_{\substack{u \in \cC_{r(n)} \\ |u| \leq n}} \e^{-V_u} (n-|u|)\Fbar(t_n - (V_u -|u|c) ).
\end{equation*}
Then we use regular variation of $\Fbar$ to end up with an approximation to Nerman's martingale which converges to the limit of Biggins' martingale.
\subsection{Proof of Theorem \ref{thm:brw-large-deviations} -- Part 1: Preparation}\label{subsec:proof-part1}
For the rest of this section, let \eqref{eq:heavy-tailed} through \eqref{eq:mu-reg-var} be satisfied, fix $a > \sqrt{p-2}$ and a sequence $(t_n)_{n \in \N}$ with $t_n \geq a \sigma \sqrt{n \log n}$ for all $n \in \N$.
Occasionally we need to distinguish the cases \eqref{item:case1} and \eqref{item:case2} from Theorem \ref{thm:brw-large-deviations}.
The proof is split into several lemmas to improve readability.

First we argue that if suffices to show that \eqref{eq:brw-large-deviations} holds in probability.
Note that the many-to-one formula (Lemma \ref{lemma:many21}) and Nagaev's theorem (Theorem \ref{thm:Nagaev}) give
\begin{equation*}
     \bbE\Big[ \frac{\Zbar_n(nc + t_n)}{n \Fbar(t_n)} \Big] = \frac{\Prob(S_n > nc + t_n)}{n \Prob(S_1 > t_n)} \convn 1 = \bbE[W].
\end{equation*}
Therefore, the claimed convergence in $L^1$ follows from convergence in probability (see e.g.~\cite[Theorem 5.12]{Kallenberg2021}).
In case \eqref{item:case2}, the claimed convergence in $L^q$ for all $q < 2 - \gamma^{-1}$ in turn follows from convergence in probability combined with Proposition \ref{prop:stochdom} below.

We define two sequences $(m(n))_{n \in \N}$ and $(r(n))_{n \in \N}$.
In case \eqref{item:case1}, $r(n)$ is already defined.
In case \eqref{item:case2} we choose for $r(n)$ any sequence with $r(n) \uparrow \infty$ and $r(n) \in o(n)\cap o(t_n)$, for instance $r(n) \coloneqq \frac{n \wedge t_n}{\log n}$.
In both cases, let
\begin{align}
     m(n) \coloneqq \Big\lceil\frac{2r(n)}{c} \Big\rceil,\quad n \in \N. \label{eq:m-def}
\end{align}
This definition is taylored to Lemma \ref{lemma:shaving} below. 
Furthermore, fix $\eps > 0$ such that $(1-\eps) a > \sqrt{p-2}$.
During the proof we will often use that due to regular variation of $\Fbar$ (by assumption \eqref{eq:mu-reg-var}) we have
\begin{equation*}
     \frac{\Fbar((1-\eps)t_n)}{\Fbar(t_n)} \conv{n}{\infty} (1-\eps)^{-p}.
\end{equation*}
Therefore, there exists a constant $K \in (0,\infty)$ such that for all $n \in \N$ we have
\begin{equation}\label{eq:mufrac-bound}
     \frac{\Fbar((1-\eps)t_n)}{\Fbar(t_n)} \leq K.
\end{equation}
To simplify notation, for individuals $u$ with $|u| \leq n$ we define
\begin{equation*}
     X_u(n) \coloneqq [\Zbar_{n-|u|}(nc+t_n-V_u)]_u = \sum_{|v| = n - |u|}\e^{-(V_{uv}-V_u)} \1\{V_{uv} > nc + t_n \}
\end{equation*}
whereas for $|u| > n$ we set $X_u(n) \coloneqq 0$.
We claim that for all $n \in \N$ with $r(n) < nc + t_n$ we have the decomposition
\begin{equation}\label{eq:Z-branching}
     \Zbar_n(nc + t_n) = \sum_{|u| = n} \e^{-V_u}\1\{V_u > nc + t_n\} = \sum_{u \in \cC_{r(n)}} \e^{-V_u} X_u(n)
\end{equation}
where $\cC_{r(n)}$ is defined in \eqref{eq:coming-gen}.
Indeed, every non-zero summand in the (implicit) double sum on the right-hand side corresponds to an individual $u$ with $|u|=n$ and $V_u > nc + t_n$ and thus belongs to the sum of the left-hand side.
Conversely, every individual $u$ with $|u| = n$ and $V_u > nc + t_n$ has an ancestor $v \leq u$ with $v \in \cC_{r(n)}$ due to $r(n) < nc + t_n$, so its corresponding summand appears in $X_v(n)$.

Using the many-to-one formula we obtain for every fixed individual $u$ with $|u| \leq n$
\begin{align}\label{eq:EX}
     \bbE[X_u(n)\,|\,V_u] &= \Fbar_{n-|u|}\big(t_n - (V_u - |u|c)\big)
\end{align}
where $\Fbar_n(t) = \Prob(S_n - nc > t)$, $t \in \R$ should be recalled.
\begin{Def}
Let $(X(n,T))_{n\in \N,T > 0}$ be a stochastic process and $f:\N \times (0,\infty) \to (0,\infty)$ be a function.
We write $X(n,T) \in \error{f(n,T)}$ if, for all $\delta > 0$,
\begin{equation}\label{eq:error}
     \lim_{T \to \infty} \limsup_{n \to \infty} \bbP\Big(\Big|\frac{X(n,T)}{f(n,T)}\Big| \geq \delta \Big) = 0.
\end{equation}
Moreover, if $(Y(n,T))_{n\in \N,T > 0}$ is another stochastic process then we write $X(n,T) = Y(n,T) + \error{f(n,T)}$ if $X(n,T) - Y(n,T) \in \error{f(n,T)}$.
\end{Def}
\begin{Rem}\label{rem:error}
(a)
Note that
\begin{equation*}
     \lim_{T \to \infty} \limsup_{n \to \infty} \bbE\Big|\frac{X(n,T)}{f(n,T)}\Big| = 0
\end{equation*}
implies $X(n,T) \in \error{f(n,T)}$ by Markov's inequality, which we will frequently use without any further comment.
\vspace{0.2cm}

\noindent (b)
If $(X_n)_{n \in \N}$ and $(Y_n)_{n \in \N}$ are sequences of random variables and we have $X_n = X(n,T) + \error{1}$ and $Y_n = X(n,T) + \error{1}$ then if follows $X_n - Y_n \to 0$ in probability as $n \to \infty$.
Indeed, for every $\delta > 0$ and $T > 0$ we have
\begin{equation*}
     \bbP(|X_n - Y_n| > \delta) \leq \bbP(|X_n - X(n,T)| > \delta/2) + \bbP(|Y_n - X(n,T)| > \delta/2).
\end{equation*}
Applying $\limsup_{n \to \infty}$ and then taking $T \to \infty$ gives the claim.
\end{Rem}
\vspace{0.2cm}
\begin{Lemma}\label{lemma:approx}
     For fixed $T > 0$ let
     \begin{equation}\label{eq:coming-gen-cutoff}
          \cC_{t}^T \coloneqq \{u \in \cC_t\,:\, V_u \leq t + T\},\quad t > 0.
     \end{equation}
     Then the following hold, with $m(n)$, $r(n)$ as above:
     \begin{align}
          \Zbar_n(nc+t_n) &= \sum_{\substack{u \in \cC_{r(n)}^T \\ |u| \leq m(n)}} \e^{-V_u} X_u(n) + \error{n\Fbar(t_n)}, \label{eq:Zhat-approx} \\
          Y_{r(n)} &= \sum_{\substack{u \in \cC_{r(n)}^T \\ |u| \leq m(n)}} \e^{-V_u} + \error{1}.\label{eq:Nerman-MG-approx}
     \end{align}
\end{Lemma}
The proof of Lemma \ref{lemma:approx} builds on the following estimates for random walks.
\begin{Lemma}\label{lemma:shaving}
     Let $(S_n)_{n \in \N}$ be the associated random walk.
     Fix $a > \sqrt{p-2}$ and a sequence $(t_n)_{n \in \N}$ such that $t_n \geq a \sigma \sqrt{n \log n}$ for all $n \in \N$.
     Further let $(r(n))_{n \in \N}, (m(n))_{n \in \N}$ be sequences that satisfy $r(n), m(n) \uparrow \infty$, $m(n) \in \N$, $m(n) \in o(n)$, $r(n) \in o(n) \cap o(t_n)$, and
     \begin{equation*}
          \limsup_{n \to \infty} \frac{r(n)}{m(n)} < c.
     \end{equation*}
     \begin{enumerate}[(a)]
          \item Let $\tau_r \coloneqq \inf\{n \in \N\,:\, S_n > r\}$, $r \in \R$. Then
          \begin{equation}\label{eq:tau-bound}
               \Prob(S_n > nc + t_n,\,\tau_{r(n)} > m(n)) \in o(n\Fbar(t_n)).
          \end{equation}
          \item 
          Let $R_s \coloneqq S_{\tau_s} - s$ be the overshoot at $s > 0$.
          Then we have
          \begin{equation}\label{eq:overshoot-bound}
               \lim_{T \to \infty} \limsup_{n \to \infty} \frac{\Prob(S_n > nc + t_n,\,R_{r(n)} > T)}{n \Fbar(t_n)} = 0.
          \end{equation}
     \end{enumerate}
\end{Lemma}
\begin{proof}[Proof of Lemma \ref{lemma:shaving}]

\noindent
(a)
As a consequence of the central limit theorem
we have
\begin{equation}\label{eq:S<r}
     \Prob(S_{m(n)} \leq r(n)) = \Prob\bigg(\frac{S_{m(n)} - m(n)c}{\sqrt{m(n)}} \leq \sqrt{m(n)}\Big(\frac{r(n)}{m(n)} - c\Big) \bigg) \convn 0.
\end{equation}
We write $\Prob\!_x$ for the law of the random walk starting from $x \in \R$, where $\Prob\!_0 = \Prob$. 
Using the Markov property at time $m(n)$ we obtain for all $n \in \N$ such that $n - m(n) > 0$
\begin{align}\label{eq:shaving-lemma-markov}
     \begin{split}
          &\Prob(S_n > nc + t_n,\,\tau_{r(n)} > m(n)) \\
          \leq\quad &\Prob(S_n > nc + t_n,\,S_{m(n)} \leq r(n)) \\  
          = \quad &\E[\1\{S_{m(n)} \leq r(n)\} \Prob\!_{S_m(n)}(S_{n - m(n)} > nc + t_n)] \\
          \leq \quad &\Prob(S_{m(n)} \leq r(n)) \Prob\!_{r(n)}(S_{n - m(n)} > nc + t_n) \\
          = \quad &\Prob(S_{m(n)} \leq r(n)) \Fbar_{n-m(n)}(m(n)c + t_n - r(n)) \\
          \leq \quad &\Prob(S_{m(n)} \leq r(n)) \Fbar_{n-m(n)}(t_n - r(n))
     \end{split}
\end{align}
Recall that we fixed $\eps \in (0,1)$ such that $a(1-\eps) > \sqrt{p-2}$.
By Theorem \ref{thm:Nagaev} there exists $n_0 \in \N$ such that for all $n \geq n_0$ and $t \geq (1-\eps)a \sigma \sqrt{n \log n}$ we have
\begin{equation}\label{eq:frac-const-bound}
     \frac{\Fbar_n(t)}{n \Fbar(t)} \leq 2.
\end{equation}
Since $r(n) \in o(t_n)$ there exists $n_1 \in \N$ such that $t_n - r(n) \geq (1-\eps)t_n$ for all $n \geq n_1$.
Therefore, for $n \geq n_1$ we have
\begin{align*}
     t_n - r(n) &\geq (1-\eps)t_n \geq (1-\eps)a \sigma \sqrt{n \log n} \\
     &\geq (1-\eps)a \sigma \sqrt{(n-m(n)) \log (n-m(n))}.
\end{align*}
Thus for $n \geq n_1$ such that $n - m(n) \geq n_0$ (which is the case for all sufficiently large $n$ due to $m(n) \in o(n)$) we find
\begin{align*}
     &\frac{\Fbar_{n-m(n)}(t_n - r(n)}{n \Fbar(t_n)} \\
     =\quad &\frac{n - m(n)}{n} \frac{\Fbar(t_n - r(n))}{\Fbar(t_n)} \frac{\Fbar_{n - m(n)}(t_n - r(n))}{(n-m(n))\Fbar(t_n - r(n))} \\
     \leq \quad &2 \frac{\Fbar((1-\eps)t_n)}{\Fbar( t_n)}.
\end{align*}
By \eqref{eq:mufrac-bound}, the right-hand side is bounded by $2 K$.
Thus we conclude that
\begin{equation*}
     \Fbar_{n-m(n)}(t_n - r(n) \in \cO(n \Fbar( t_n)).
\end{equation*}
Combining this with \eqref{eq:S<r}, we infer \eqref{eq:tau-bound} from \eqref{eq:shaving-lemma-markov}.

\noindent
(b)
We have
\begin{align}
     \begin{split}\label{eq:overshoot-split}
          &\Prob(S_n > nc + t_n,\,R_{r(n)} > T) \\
          =\quad &\Prob(S_n > nc + t_n,\,R_{r(n)} > T,\,\tau_{r(n)} > m(n)) \\
          +\quad &\Prob(S_n > nc + t_n,\,R_{r(n)} > T,\,\tau_{r(n)} \leq m(n),\,S_{\tau_{r(n)}} > \eps t_n) \\
          +\quad &\Prob(S_n > nc + t_n,\,R_{r(n)} > T,\,\tau_{r(n)} \leq m(n),\,S_{\tau_{r(n)}} \leq \eps t_n).
     \end{split}
\end{align}
From (a) we infer that the first summand in the right-hand side is in $o(n \Prob(S_1 > t_n))$.
Note that since $S_{\tau_{r(n)}} > \eps t_n$ implies $R_{r(n)} > T$ for all but finitely many $n$, the second summand on the right-hand side of \eqref{eq:overshoot-split} is bounded by
\begin{equation*}
     \Prob(\tau_{r(n)} \leq m(n),\,S_{\tau_{r(n)}} > \eps t_n).
\end{equation*}
Further, we find
\begin{align*}
\Prob(\tau_{r(n)} \leq m(n),\,S_{\tau_{r(n)}} > \eps t_n) &\leq \sum_{j=1}^{m(n)}\E[\1\{S_{j-1} \leq r(n)\}\Prob_{S_{j-1}}(S_1 > \eps t_n)] \\
     &\leq \Fbar(\eps t_n - r(n)) m(n).
\end{align*}
Since $\Fbar$ is regularly varying and we have $r(n) \in o(t_n)$ and $m(n) \in o(n)$ this is in $o(n \Fbar(t_n))$.

Finally, using the Markov property the third summand on the right-hand side of \eqref{eq:overshoot-split} is given by
\begin{align*}
     &\sum_{j=1}^{m(n)}\E[\1\{\tau_{r(n)} = j,\,r(n) + T < S_j < \eps t_n\} \Prob_{S_j}(S_{n-j} > nc + t_n)] \\
     = \quad &\sum_{j=1}^{m(n)}\E[\1\{\tau_{r(n)} = j,\,r(n) + T < S_j < \eps t_n\} \Fbar_{n-j}(t_n + jc - S_j)].
\end{align*}
If $S_j \leq \eps t_n$ then by \eqref{eq:frac-const-bound} and \eqref{eq:mufrac-bound} we have for all $n$ such that $n-m(n) \geq n_0$
\begin{equation*}
     \Fbar_{n-j}(t_n + jc - S_j) \leq \Fbar_{n-j}((1-\eps)t_n) \leq 2 (n-j)\Fbar((1-\eps)t_n) \leq 2Kn\Fbar(t_n).
\end{equation*}
Therefore,
\begin{align*}
     &\sum_{j=1}^{m(n)}\E[\1\{\tau_{r(n)} = j,\,S_j > r(n) + T,\,S_j\leq \eps t_n\} \Fbar_{n-j}(t_n + jc - S_j)] \\
     \leq \quad &2Kn\Fbar(t_n)\Prob(R_{r(n)} > T).
\end{align*}
Since $R_{r(n)}$ converges in distribution as $n \to \infty$ (see e.g. \cite[Lemma 12.22]{Kallenberg2021}), we get
\begin{equation*}
     \lim_{T \to \infty} \limsup_{n \to \infty} \frac{\Prob(S_n > nc + t_n,\,R_{r(n)} > T,\,\tau_{r(n)} \leq m(n),\,S_{\tau_{r(n)}} \leq \eps t_n)}{n \Fbar(t_n)} = 0.
\end{equation*}
This concludes the proof of (b).
\end{proof}

\begin{proof}[Proof of Lemma \ref{lemma:approx}.]
Let $(\G_n)_{n \in \N}$ be the natural filtration of $(S_n)_{n \in \N}$.
By conditioning on $\F_{\cC_{r(n)}}$ and then using \eqref{eq:EX} we obtain, for all $T > 0$,
\begin{align*}
     &\bbE\bigg[ \sum_{u \in \cC_{r(n)}} \e^{-V_u} X_u(n) \1\{V_u > r(n) + T\} \bigg] \\
     =\quad &\bbE\bigg[ \sum_{u \in \cC_{r(n)}} \e^{-V_u} \Fbar_{n - |u|}\big( t_n-(V_u - |u|c)\big) \1\{V_u > r(n) + T,\,|u| \leq n\} \bigg].
\end{align*}
Now we apply the many-to-one formula for stopping lines \eqref{eq:many214lines} and get
\begin{align*}
     &\bbE\bigg[ \sum_{u \in \cC_{r(n)}} \e^{-V_u} \Fbar_{n - |u|}\big( t_n-(V_u - |u|c)\big) \1\{V_u > r(n) + T,\,|u| \leq n\} \bigg] \\
     =\quad &\E\big[ \Fbar_{n - \tau_{r(n)}}\big( t_n-(S_{\tau_{r(n)}} - \tau_{r(n)}c)\big) \1\{S_{\tau_{r(n)}} > r(n) + T,\,\tau_{r(n)} \leq n\} \big].
\end{align*}
Using the strong Markov property, this equals
\begin{align*}
     &\E\big[ \Prob(S_n > nc + t_n \,|\,\G_{\tau_{r(n)}}) \1\{S_{\tau_{r(n)}} > r(n) + T,\,\tau_{r(n)} \leq n\} \big] \\
     =\quad&\Prob(S_n > nc + t_n,\,S_{\tau_{r(n)}} > r(n) + T,\,\tau_{r(n)} \leq n).
\end{align*}
where $\G_{\tau_{r(n)}}$ denotes the pre-$\tau_{r(n)}$ $\sigma$-algebra.
With $R_t = S_{\tau(t)} - t$ it follows
\begin{align*}
     \bbE\bigg[ \sum_{u \in \cC_{r(n)}} \e^{-V_u} X_u(n) \1\{V_u > r(n) + T\} \bigg] &= \Prob(S_n > nc + t_n,\,R_{r(n)} > T,\,\tau_{r(n)} \leq n) \\
     &\leq \Prob(S_n > nc + t_n,\, R_{r(n)} > T).
\end{align*}
Thus with Lemma \ref{lemma:shaving} (b) and \eqref{eq:Z-branching} we infer (c.f.~Remark \ref{rem:error} (a))
\begin{equation*}
     \Zbar_n(nc+t_n) = \sum_{u \in \cC_{r(n)}} \e^{-V_u} X_u(n)\1\{V_u \leq r(n) + T\} + \error{n\Fbar(t_n)}
\end{equation*}
To restrict the sum further, we infer similarly
\begin{align*}
     \bbE\bigg[ \sum_{u \in \cC_{r(n)}} \e^{-V_u} X_u(n) \1\{m(n)<|u|\} \bigg] = \Prob(S_n > nc + t_n,\, m(n) < \tau_{r(n)} \leq n)
\end{align*}
which by Lemma \ref{lemma:shaving} (a) is in $o(n\Fbar(t_n))$. 
Therefore, we obtain \eqref{eq:Zhat-approx}.

Analogously we approximate Nerman's martingale $Y_t = \sum_{u \in \cC_t}\e^{-V_u}$.
Indeed, we have 
\begin{align*}
     \bbE\bigg[ \sum_{u \in \cC_{r(n)}}\e^{-V_u}\1\{V_u > r(n) + T\} \bigg] &= \Prob(R_{r(n)} > T),\\
     \bbE\bigg[ \sum_{u \in \cC_{r(n)}}\e^{-V_u}\1\{m(n) < |u|\} \bigg] &= \Prob(\tau_{r(n)} > m(n)) \leq \Prob(S_{m(n)} < r(n)).
\end{align*}
By \eqref{eq:S<r} and the fact that the overshoot $R_t$ converges in distribution as $t \to \infty$, both integrands on the left-hand side are in $\error{1}$.
Thus we conclude \eqref{eq:Nerman-MG-approx}.
\end{proof}
\subsection{Proof of Theorem \ref{thm:brw-large-deviations} -- Part 2: A weak law of large numbers}\label{subsec:proof-part2}
The goal of this section is to prove a weak law of large numbers for the right-hand side of \eqref{eq:Zhat-approx}.
\begin{Lemma}\label{lemma:Zhat-wlln}
     For all sufficiently large $T > 0$ we have
     \begin{equation}\label{eq:Zhat-LLN}
          \frac{1}{n \Fbar(t_n)} \sum_{\substack{u \in \cC_{r(n)}^T \\ |u| \leq m(n)}} \e^{-V_u} \big( X_u(n) - \bbE[X_u(n)\,|\,V_u] \big) \pconv{n}{\infty} 0.
     \end{equation}
\end{Lemma}
We prove the cases \eqref{item:case1} and \eqref{item:case2} of Theorem \ref{thm:brw-large-deviations} separately, starting with case \eqref{item:case2}.
\subsubsection{Proof of Lemma \ref{lemma:Zhat-wlln} in case \eqref{item:case2}}
To ensure stochastic domination for a weak law of large numbers in this case we use the following proposition.
Its proof is more involved and therefore postponed to Section \ref{sec:stochdom}.
\begin{Prop}\label{prop:stochdom}
     Suppose that \eqref{eq:heavy-tailed} through \eqref{eq:Zhat-tail-gamma} hold with $\gamma \in (1,2)$.
     Define $\eta \coloneqq 1 - \gamma^{-1}$ and fix $a > \sqrt{p-2}$.
     Then for all $(t_n)_{n \in \N}$ with $t_n \geq a \sigma \sqrt{n \log n}$ for all $n \in \N$ we have
     \begin{equation*}
          \bbE[\Zbar_n(nc + t_n)^{1+\eta}] \in \cO\big( (n \Fbar(t_n))^{1 + \eta} \big)\quad \text{ as }n \to \infty.
     \end{equation*}
\end{Prop}
\begin{proof}[Proof of Lemma \ref{lemma:Zhat-wlln} in the case \eqref{item:case2}, admitting Proposition \ref{prop:stochdom}.]
From \cite[Theorem 3]{Biggins1998} we know that assumption \eqref{eq:heavy-tailed} implies
\begin{equation}\label{eq:min-infty}
     \lim_{n \to \infty} \min_{|u| = n} V_u = \infty
\end{equation}
almost surely on survival, whence $N(t) \coloneqq \sum_{u \in \I} \1\{V_u \leq t\}$ is finite almost surely for all $t \geq 0$ (here we use that the point process $\xi$ is almost surely locally finite).
Since $|\cC_{t}^T| \leq N(t + T)$, it follows that $\cC_{t}^T$ is almost surely finite for all $T > 0$.
Let $T > 0$ satisfy $\mu(-\infty, T] > 1$ (this is true for all sufficiently large $T$ by \eqref{eq:heavy-tailed}).
From the proof of \cite[Lemma 3.6]{Nerman1981} we get that for some $c > 0$ we have almost surely on survival
\begin{equation*}
     \liminf_{t \to \infty}\frac{|\cC_{t}^T|}{N(t)} \geq c.
\end{equation*}
As $N(t) \uparrow \infty$ it follows $|\cC_{t}^T| \to \infty$ as $t \to \infty$ almost surely on survival.

We shall use \cite[Proposition 3.8]{Nerman1981} to infer
\begin{equation}\label{eq:lln-case2}
     \frac{1}{|\cC_{r(n)}^T|} \sum_{\substack{u \in \cC_{r(n)}^T \\ |u| \leq m(n)}} \frac{\e^{r(n)-V_u}}{n \Fbar(t_n)} \big( X_u(n) - \bbE[X_u(n)\,|\,V_u] \big) \pconv{n}{\infty} 0.
\end{equation}
Note that the summands are independent given $\F_{\cC_{r(n)}}$ (see \cite[Theorem 4.14]{Jagers1989}), and $\e^{r(n)-V_u} \leq 1$ for all $u \in \cC_{r(n)}^T$.
To conclude \eqref{eq:lln-case2} it suffices to show that
\begin{equation*}
     \frac{X_u(n)}{n\Fbar(t_n)}, \quad u \in \cC_{r(n)}^T
\end{equation*}
is uniformly stochastically dominated conditional on $\F_{\cC_{r(n)}}$ for all $n \in \N$ by some integrable random variable.
This is turn follows if we show that there exist $C \in (0,\infty)$ and $\eta > 0$ such that for all $n \in \N$ we have
\begin{equation}\label{eq:X-uniform-moment}
     \sup_{u \in \cC_{r(n)}^T} \bbE\Big[ \Big(\frac{X_u(n)}{n\Fbar(t_n)} \Big)^{1+\eta}\,|\,\F_{\cC_{r(n)}}\Big] \leq C.
\end{equation}
Indeed, if a family of non-negative random variables $(X_i)_{i \in I}$ has $\sup_{i \in I} \bbE[X_i^{1+\eta}] \leq C$, then by Markov's inequality we have for all $i \in I$
\begin{equation*}
     \bbP(X_i > t) \leq \frac{C}{t^{1+\eta}} \wedge 1,\quad t > 0.
\end{equation*}
Therefore, if the right-hand side is taken to define the tail probability of a random variable $Y$, then $(X_i)_{i \in I}$ is uniformly stochastically dominated by $Y$ and $Y$ is integrable.
Thus we see that \eqref{eq:X-uniform-moment} is sufficient to show that $X_u(n)/n \Fbar(t_n)$, $u \in \cC_{r(n)}^T$ is stochastically dominated.

Recall that $\eps > 0$ is chosen such that $(1-\eps)a > \sqrt{p-2}$. 
By Proposition \ref{prop:stochdom} there exist $\eta > 0$ and $C \in (0,\infty)$ such that
\begin{equation*}
     \sup_{n \in \N} \bbE\Big[ \Big( \frac{\Zbar_n(nc+(1-\eps)t_n)}{n \Fbar((1-\eps)t_n)} \Big)^{1 + \eta}\Big] \leq C.
\end{equation*}
Moreover, by $r(n) \in o(t_n)$ there exists $n_0 \in \N$ such that for all $n \geq n_0$ we have
\begin{equation*}
     t_n - (r(n) + T) \geq (1-\eps) t_n. 
\end{equation*}
Then for all $n \geq n_0$ and $u \in \cC_{r(n)}^T$ with $|u| \leq n$ we find, due to $V_u \leq r(n) + T$ and $t \mapsto \Zbar_j(t)$ being decreasing for all $j \in \N$,
\begin{align*}
     X_u(n) = [\Zbar_{u - |u|}(nc + t_n - V_u)]_u &\leq [\Zbar_{u - |u|}((n-|u|)c + t_n - (r(n) + T))]_u \\
     &\leq [\Zbar_{u - |u|}((n-|u|)c + (1-\eps)t_n)]_u.
\end{align*}
Now, since $[\Zbar_{j}(t)]_u$ conditional on $\F_{\cC_{r(n)}}$ has the same law as $\Zbar_{j}(t)$ for all $j \in \N$, $t \in \R$ and $u \in \cC_{r(n)}$, we get for all $n \geq n_0$ and $u \in \cC_{r(n)}^T$ with $|u| < n$ (if $|u| = n$ then $X_u(n) = 0$)
\begin{align*}
     &\bbE\Big[ \Big(\frac{X_u(n)}{n\Fbar(t_n)} \Big)^{1+\eta}\,|\,\F_{\cC_{r(n)}}\Big] \\
     \leq \quad &\Big( \frac{\Fbar((1-\eps)t_n)}{\Fbar(t_n)} \Big)^{1+\eta} \bbE\Big[ \Big( \frac{\Zbar_{n-|u|}((n-|u|)c+(1-\eps)t_n)}{(n-|u|) \Fbar((1-\eps)t_n)} \Big)^{1 + \eta}\Big] \\
     \leq \quad &K^{1+\eta} C
\end{align*}
where in the last step we used \eqref{eq:mufrac-bound}.
This shows \eqref{eq:X-uniform-moment}.
Thus we conclude \eqref{eq:lln-case2} using \cite[Proposition 3.8]{Nerman1981} (note that restricting the sum further to $u \in \cC_{r(n)}^T$ with $|u| \leq m(n)$ has no effect).
Borrowing another trick from \cite{Nerman1981} we see that, for all $n \in \N$,
\begin{equation}\label{eq:coming-gen-upper-bound}
     \e^{-(r(n)+T)} |\cC_{r(n)}^T| \leq \sum_{u \in \cC_{r(n)}^T} \e^{-V_u} \leq Y_{r(n)}.
\end{equation}
Since $Y_{r(n)}$ converges almost surely by Theorem \ref{thm:Nermans-MG}, the claim follows.
\end{proof}
\subsubsection{Proof of Lemma \ref{lemma:Zhat-wlln} in case \eqref{item:case1}}
In the case \eqref{item:case1} we stochastically bound $X_u(n)$ instead of $X_u(n)/n \Fbar(t_n)$ and use a suitable Marcinkiewicz-Zygmund-type weak law of large numbers to get convergence in probability with a rate.
Its proof is a combination of the proof of \cite[Proposition 4.1]{Nerman1981} with a Marcinkiewicz-Zygmund weak law of large numbers.
We present the details in Appendix \ref{sec:wlln-proof}.
\begin{Thm}\label{thm:wlln-w/growth}
     Let $(X_{kj})_{j \leq n_k,\,k \in \N}$ be a family of non-negative random variables such that $X_{k1},\ldots,X_{kn_k}$ is independent for all $k \in \N$.
     Further assume that there exists a non-negative random variable $X$ such that 
     \begin{enumerate}[(i)]
          \item $X_{kj} \preceq X$ for all $k,j$,
          \item $\bbP(X > t) \in o(t^{-p})$ as $t \to \infty$ for some $p \in (1,2)$,
          \item $\limsup_{k \to \infty} n_k^{-q}\sum_{j \leq k} n_j < \infty$ for some $q > 0$.
     \end{enumerate}
     Then we have
     \begin{equation*}
          \frac{1}{n_k^{q/p}} \sum_{j=1}^{n_k} \big( X_{kj} - \bbE[X_{kj}] \big) \pconv{k}{\infty} 0.
     \end{equation*}
\end{Thm}
In our case we have $n_k = |\cC^T_{r(k)}|$ with $\cC^T_{r(k)}$ from \eqref{eq:coming-gen-cutoff}.
To get condition (iii) of the above theorem is the purpose of the following lemma.
\begin{Lemma}\label{lemma:coming-gen-rate}
     Let $(r(n))_{n \in \N}$ satisfy \eqref{eq:r-growth} with $\delta \in (0,\gamma-1)$.
     Then for all $\delta' > \delta$ and all sufficiently large $T > 0$ we have, almost surely on survival,
     \begin{equation*}
          \limsup_{n \to \infty} |\cC_{r(n)}^T|^{-(1 + \delta')} \sum_{j=1}^{n-1} |\cC_{r(n)}^T| < \infty.
     \end{equation*}
\end{Lemma}
\begin{proof}
From \eqref{eq:coming-gen-upper-bound} above we know that for all $T>0$ there exists a finite random variable $Y$ such that, almost surely,
\begin{equation*}
     |\cC_t^T| \leq Y \e^t \quad \text{ for all }t \geq 0.
\end{equation*}
We also need to establish a lower bound.
To that end consider the point process
\begin{equation*}
     \tilde{\xi} \coloneqq \sum_{u \in \cC_0} \delta_{V_u}
\end{equation*}
(with $\cC_0$ from \eqref{eq:coming-gen}) which is almost surely supported in $[0,\infty)$.
It is well-known (see \cite{Biggins1997a,Biggins2005,Alsmeyer2012}) that the branching random walk $(V_u)_{u \in \I}$ has an embedded instance of a branching random walk with reproduction point process $\tilde{\xi}$, i.e., there exists a random injective map $\iota:\I \to \I$ such that $(\tilde{V}_u)_{u \in \I} \coloneqq (V_{\iota(u)})_{u \in \I}$ is a branching random walk with reproduction point process $\tilde{\xi}$.
An individual $u \in \I$ appears in the embedded branching random walk if and only if $V_u > V_v$ for all $v < u$ (i.e., its position is maximal among the positions of all its ancestors or, in other words, its position is a strictly ascending ladder height in the sequence of positions in its lineage).
From \eqref{eq:min-infty} it follows that the embedded process survives if and only if the original one survives.
In particular the embedded branching process is supercritical.
Therefore, for all sufficiently large $T > 0$ we have 
\begin{equation*}
     \tilde{\mu}[0,T] \coloneqq \bbE\big[\tilde{\xi}[0,T]\big] > 1.
\end{equation*}
Fix one such $T > 0$.
Using \eqref{eq:many214lines} we infer (c.f.~\cite[Theorem 10]{Biggins2005})
\begin{equation*}
     \bbE\bigg[\int \e^{-x} \tilde{\xi}(\d x) \bigg] = \bbE\bigg[ \sum_{u \in \cC_0} \e^{-V_u}\bigg] = \Prob(\tau_0 < \infty) = 1,
\end{equation*}
so the embedded branching process has Malthusian parameter $1$.
Moreover, it is easy to see that 
\begin{equation*}
     \tilde{\cC}_t \coloneqq \{u \in \I\,|\, \max_{v < u}\tilde{V}_v \leq t < \tilde{V}_u\}.
\end{equation*}
satisfies $\iota(\tilde{\cC}_t) = \cC_t$ (cf.~\cite[Lemma 8.1]{Biggins1997a}).
We similarly define $\tilde{\cC}_t^T$ and get $|\cC_t^T| = |\tilde{\cC}_t^T|$.
From \cite[Lemma 3.6]{Nerman1981} we infer that there exists $c > 0$ such that
\begin{equation*}
     \liminf_{t \to \infty}\frac{|\tilde{\cC}_t^T|}{\tilde{N}(t)} \geq c
\end{equation*}
where $\tilde{N}(t) \coloneqq \sum_{u \in \I} \1\{\tilde{V}_u \leq t\}$ (here we use $\tilde{\mu}[0,T] > 1$).
Now we invoke \cite[Theorem 2]{Biggins1995} to get\footnote{Note that the cited reference contains a non-lattice assumption which it inherits from Nerman's work \cite{Nerman1981}. It is straight-forward to obtain the same result in the lattice case via the same proof, using \cite{Gatzouras2000} instead of \cite{Nerman1981}.}
\begin{equation*}
     \liminf_{t \to \infty} \frac{\log \tilde{N}(t)}{t} \geq 1 \quad \text{almost surely on survival.}
\end{equation*}
Conclusively, we have
\begin{equation*}
     \liminf_{t \to \infty} \frac{\log |\cC_t^T|}{t} \geq 1 \quad \text{almost surely on survival.}
\end{equation*}
Let $\delta' > \delta$ and take $\eta \in (0,1)$ such that $(1-\eta)(1+\delta') > 1+\delta$.
Almost surely on survival, for all sufficiently large $t$ we have $|\cC_t^T| \geq \e^{(1-\eta) t}$, so that for all sufficiently large $n$
\begin{equation*}
     |\cC_{r(n)}^T|^{-(1+\delta')} \leq \e^{-(1-\eta)(1+\delta')r(n)} \leq \e^{-(1+\delta)r(n)}.
\end{equation*}
Therefore, the claim follows from \eqref{eq:r-growth}.
\end{proof}
\begin{proof}[Proof of Lemma \ref{lemma:Zhat-wlln} in the case \eqref{item:case1}.]
Let $\delta \in (0,\gamma -1)$ and $q \in (1+\delta,\gamma)$ satisfy \eqref{eq:r2t}.
Fix $q' \in (q,\gamma)$ and let $\delta' > \delta$ satisfy $(1+\delta')/q' = (1+\delta)/q$.
We will use Theorem \ref{thm:wlln-w/growth} to show that
\begin{equation}\label{eq:wlln-w/growth-application}
     |\cC_{r(n)}^T|^{-\frac{1+\delta'}{q'}} \sum_{u \in \cC_{r(n)}^T} \e^{r(n)-V_u} \big( X_u(n) - \bbE[X_u(n)\,|\,\F_{\cC_{r(n)}}] \big) \pconv{n}{\infty} 0.
\end{equation}
To that end we have to stochastically dominate each summand $X_u(n)$, $u \in \cC_{r(n)}$ conditional on $\F_{\cC_{r(n)}}$ uniformly by an integrable random variable $Y$ that satisfies $\Prob(Y > t) \in o(t^{-q'})$.
Similar to the proof in case \eqref{item:case2} (see \eqref{eq:X-uniform-moment}) it suffices to show that, for some $C \in (0,\infty)$ and all $n \in \N$ we have
\begin{equation}\label{eq:X-uniform-moment-case1}
     \sup_{u \in \cC_{r(n)}^T}\bbE[X_u(n)^\gamma\,|\, \F_{\cC_{r(n)}}] \leq C.
\end{equation}
As the tail of the dominating random variable $Y$ we then take $\Prob(Y > t) = C t^{-\gamma} \wedge 1$, which is in $o(t^{-q'})$ as claimed.
For every $t \in \R$ and $n \in \N$ we have
\begin{equation*}
     \Zbar_n(t, \infty) \leq \Zbar_n(\R) = W_n,
\end{equation*}
where $(W_n)_{n \in \N_0}$ is Biggins' martingale, see \eqref{eq:Biggins-MG}.
By assumption \eqref{eq:Lp-cond} and \cite[Thm 2.1]{Liu2000} we have $W_n \to W$ in $\L^\gamma$,
which implies that
\begin{equation*}
     \sup_{n \in \N_0} \bbE[W_n^\gamma] \leq \bbE[W^\gamma] < \infty.
\end{equation*}
Since for each $u \in \cC_{r(n)}$ the law of $X_u(n)$ given $\F_{\cC_{r(n)}}$ is that of $\Zbar_{n - |u|}(nc+t_n - V_u)$, we arrive at \eqref{eq:X-uniform-moment-case1}.
Thus we have established (i) and (ii) of Theorem \ref{thm:wlln-w/growth}, while (iii) follows from Lemma \ref{lemma:coming-gen-rate}.
Therefore, the theorem yields \eqref{eq:wlln-w/growth-application}.
Finally, we find
\begin{align*}
     &\frac{1}{n \Fbar(t_n)} \sum_{u \in \cC_{r(n)}^T} \e^{-V_u}\big( X_u(n) - \bbE[X_u(n)\,|\,V_u] \big) \\
     =\quad &\frac{\e^{-r(n)(1 - (1+\delta)/q)}}{n \Fbar(t_n)} \big(\e^{-r(n)}|\cC_{r(n)}^T|\big)^\frac{1+\delta'}{q'} \\
     \times\  &|\cC_{r(n)}^T|^{-\frac{1+\delta'}{q'}} \sum_{u \in \cC_{r(n)}^T} \e^{r(n)-V_u} \big( X_u(n) - \bbE[X_u(n)\,|\,V_u] \big)
\end{align*}
Using \eqref{eq:r2t}, \eqref{eq:coming-gen-upper-bound}, and Theorem \ref{thm:Nermans-MG} we conclude \eqref{eq:Zhat-LLN}.
\end{proof}
\subsection{Proof of Theorem \ref{thm:brw-large-deviations} -- Part 3: Conclusion}\label{subsec:proof-part3}
Now we can finish the proof of our main result.
\begin{proof}[Proof of Theorem \ref{thm:brw-large-deviations}]
Fix $T > 0$ such that \eqref{eq:Zhat-LLN} holds.
Combining Lemma \ref{lemma:approx} with Lemma \ref{lemma:Zhat-wlln} and \eqref{eq:EX} yields
\begin{align*}
     \Zbar(t_n + nc) &= \sum_{\substack{u \in \cC_{r(n)}^T \\ |u| \leq m(n)}} \e^{-V_u} X_u(n) + \error{n\Fbar(t_n)} \\
     &= \sum_{\substack{u \in \cC_{r(n)}^T \\ |u| \leq m(n)}} \e^{-V_u} \bbE[X_u(n)\,|\,V_u] + \error{n\Fbar(t_n)} \\
     &= \sum_{\substack{u \in \cC_{r(n)}^T \\ |u| \leq m(n)}} \e^{-V_u} \Fbar_{n-|u|}\big(t_n - (V_u - |u|c)\big) + \error{n\Fbar(t_n)}.
\end{align*}
Therefore, by \eqref{eq:Nerman-MG-approx} and Remark \ref{rem:error} (b) it suffices show that
\begin{equation}\label{eq:last-step}
     \sum_{\substack{u \in \cC_{r(n)}^T \\ |u| \leq m(n)}} \e^{-V_u} \Fbar_{n-|u|}\big(t_n - (V_u - |u|c)\big) = n \Fbar(t_n) \sum_{\substack{u \in \cC_{r(n)}^T \\ |u| \leq m(n)}} \e^{-V_u} + \error{n\Fbar(t_n)}.
\end{equation}
Abbreviate the left-hand side by $\cS_{n,T}$ and define
\begin{equation*}
     Y_{n,T} \coloneqq \sum_{\substack{u \in \cC_{r(n)}^T \\ |u| \leq m(n)}} \e^{-V_u}.
\end{equation*}
Let $\delta \in (0,1)$ such that $(1-\delta)a > \sqrt{p-2}$.
By Theorem \ref{thm:Nagaev} there exists $n_0$ such that for all $n \geq n_0$
\begin{equation*}
     \sup_{x \geq a(1-\delta) \sigma \sqrt{n \log n}}\Big|\frac{\Fbar_n(x)}{n \Fbar(x)} - 1 \Big| < \delta.
\end{equation*}
Recall that for all $u \in \Lambda_n^T \cap \cC_{r(n)}$ we have $V_u \leq r(n) + T$.
Since $r(n) \in o(t_n)$, we have for all but finitely many $n$
\begin{equation}\label{eq:Fbar-arg-bound}
     t_n-(V_u - |u|c) \geq t_n - (r(n) + T) \geq (1-\delta)t_n.
\end{equation}
Therefore, for all large $n$ such that $n-m(n) \geq n_0$ we have
\begin{align*}
     \cS_{n,T} &\leq \sum_{\substack{u \in \cC_{r(n)}^T \\ |u| \leq m(n)}} \e^{-V_u} \Fbar_{n - |u|}\big( (1-\delta)t_n \big) \\
     &\leq (1+\delta) n \Fbar((1-\delta)t_n) Y_{n,T}.
\end{align*}
Thus we get for all $\eta > 0$
\begin{align*}
     \bbP\bigg( \frac{\cS_{n,T}}{n \Fbar(t_n)} - Y_{n,T} > \eta \bigg) &\leq \bbP\bigg( Y_{n,T} \Big((1+\delta) \frac{\Fbar\big((1-\delta)t_n\big)}{\Fbar(t_n)} - 1 \Big) > \eta \bigg) \\
     &\leq \bbP\bigg( Y_{r(n)} \Big((1+\delta) \frac{\Fbar\big((1-\delta)t_n\big)}{\Fbar(t_n)} - 1 \Big) > \eta \bigg).
\end{align*}
In view of Theorem \ref{thm:Nermans-MG} and \eqref{eq:mu-reg-var}, as $n \to \infty$ the right-hand side converges to
\begin{equation*}
     \limsup_{n \to \infty}\bbP\bigg( \frac{\cS_{n,T}}{n \Fbar(t_n)} - Y_{n,T} > \eta \bigg) \leq \bbP\big(W \big((1+\delta) (1-\delta)^{-p} - 1\big) > \eta \big)
\end{equation*}
which goes to zero as $\delta \downarrow 0$.
Hence we obtain
\begin{equation*}
     \bbP\bigg( \frac{\cS_{n,T}}{n \Fbar(t_n)} - Y_{n,T} > \eta \bigg) \convn 0.
\end{equation*}
To get the lower limit we use 
\begin{equation*}
     t_n - (V_u - |u|c) \leq t_n + m(n)c \leq (1-\delta)t_n
\end{equation*}
for all but finitely many $n$ due to $m(n) \in o(t_n)$.
This gives, as above,
\begin{equation*}
     \cS_{n,T} \geq (1-\delta) n \Fbar((1+\delta)t_n) Y_{n,T}.
\end{equation*}
The same reasoning as above shows that
\begin{equation*}
     \bbP\bigg( \frac{\cS_{n,T}}{n \Fbar(t_n)} - Y_{n,T} < -\eta \bigg) \convn 0.
\end{equation*}
Thus \eqref{eq:last-step} follows and the proof is complete.
\end{proof}
\section{Proof of Proposition \ref{prop:stochdom}}\label{sec:stochdom}
In this section we prove Proposition \ref{prop:stochdom}.
The proof rests on the spinal decomposition theorem which we will set up now.
For a more detailed account the reader is referred to \cite{Lyons1997,Biggins2004,Shi2015}.

Recall that we label particles in a branching random walk by elements of the Ulam-Harris tree $\I = \bigcup_n \N^n$ (see Section \ref{subsec:setup}).
Therefore, the positions $(V_u)_{u \in \I}$ of a branching random walk define a measure on the space of marked trees $\I^\ast \coloneqq (\R\cup \{\infty\})^\I$.
We further consider the set $\partial \I$ of infinite lines of descent in the Ulam-Harris tree, which we formally may identify as the set of sequences $(u_n)_{n \in \N_0} \subseteq \I$ such that $u_0 = \varnothing$ and for all $n \in \N_0$ there exists $j \in \N$ with $u_{n+1} = u_nj$.
Recall that $\F_n$ is the $\sigma$-algebra of the positions $(V_u)_{|u| \leq n}$ of a branching random walk up to generation $n \in \N_0$.
We use Biggins' martingale to obtain a probability measure $\cQ$ on $\I^\ast$ with
\begin{equation}\label{eq:Q-def}
     \frac{\d \cQ}{\d \bbP}\bigg|_{\F_n} = W_n,\quad n \in \N_0.
\end{equation}
The corresponding expectation is denoted by $\bbE^\cQ$.
It turns out that $\cQ$ is also the law of the positions of a certain \emph{spinal branching random walk}, which we will describe in the following.

We consider two types of particles which we call \emph{spine} and \emph{off-spine}, respectively.
Off-spine particles displace their offspring according to the original point process $\xi$, while spine particles do so according to the size-biased point process $\hat{\xi}$, whose law has the Radon-Nikod\'ym density $W_1$ with respect to the law of $\xi$ under $\bbP$.
The process starts with a spine particle denoted as $w_0 \coloneqq \varnothing$.
The positions $(V_u)_{u \in \cal{I}}$ in this multi-type branching random walk are constructed as usual (see Section \ref{subsec:setup}).
For every $n \in \N_0$, the next spine particle $w_{n+1}$ is chosen among the direct offspring $(w_nj)_{j \in \N}$ of $w_n$ such that $w_nj = w_{n+1}$ with probability proportional to $\e^{-(V(w_nj) - V(w_n))}$, $j \in \N$.
All other offspring particles of $w_n$ are declared off-spine.

Let $\B$ denote the law of $((V_u)_{u \in \I}, (w_n)_{n \in \N_0})$, which is a probability measure on $\I^\ast \times \partial \I$.
We continue to use the notation $\F_n$ for the $\sigma$-algebra of the positions up to generation $n \in \N_0$ on this space as well.
Note that $\F_n$ contains no information about the spine.
Further, let $\G_n$ denote the $\sigma$-algebra of the spine $(w_j)_{j \leq n}$, the positions $(V(w_j))_{j \leq n}$ and the positions of all its immediate offspring $(V(w_nj)-V(w_n))_{j \in \N}$, $n \in \N_0$.
For $u \in \I$ and $j \in \N$ we write
\begin{equation*}
     \Omega(uj) \coloneqq \{ui\,:\,j \neq i \in \N\} \subseteq \I
\end{equation*}
for the set of \emph{siblings} of $uj$. 
\begin{Thm}[Spinal decomposition theorem]\label{thm:spinal-decomp}
     With $\cQ$ and $\B$ as above, $\cQ$ is equal to the marginal law of $\B$ on $I^\ast$.
     As customary in literature we identify $\cQ$ with $\B$, so we use $\cQ$ as a probability measure on $\I^\ast \times \partial \I$.
     Furthermore, the following hold.
     \begin{enumerate}[(i)]
          \item
          Conditionally on $\G_n$, for all $n \in \N$ the processes
          \begin{equation*}
               (V(uv)-V(u))_{v \in \I},\quad u \in \Omega(w_n)
          \end{equation*}
          are independent and have the distribution of a branching random walk with reproduction point process $\xi$.
          \item
          For every $u \in \I$ with $|u| = n \in \N_0$ we have 
          \begin{equation*}
               \cQ(w_n = u\,|\,\F_n) = \frac{\e^{-V_u}}{W_n}.
          \end{equation*}
          \item
          The law of $(V(w_n))_{n \in \N_0}$ under $\cQ$ is that of the associated random walk of the branching random walk $(V_u)_{u \in \I}$ (see Section \ref{subsec:preliminaries}).
     \end{enumerate}
\end{Thm}
With the spinal decomposition theorem at hand we can now prove Proposition \ref{prop:stochdom}.
\begin{proof}[Proof of Proposition \ref{prop:stochdom}.]
We start with some preliminary observations.
Recall that $\eta = 1 - \gamma^{-1}$ and note that $1 + \eta \leq \gamma$.
Therefore, by \eqref{eq:Lp-cond} we have $\bbE[W_1^{1+\eta}] < \infty$, and \eqref{eq:Zhat-tail-gamma} combined with Jensen's inequality gives
\begin{equation}\label{eq:Zhat-moment-tail}
     \bbE[\Zbar_1(t)^{1+\eta}] \leq \bbE[\Zbar_1(t)^\gamma]^{(1+\eta)/\gamma} \in \cO(\Fbar(t)^{1+\eta}).
\end{equation}
Moreover, for any function $f$ that is regularly varying at infinity we have $\cO(f(t)) = \cO(f(\lambda t))$ for all $\lambda > 0$.
Since by \eqref{eq:mu-reg-var} for any $q > 0$ the function $t \mapsto \Fbar(t)^{q}$ is regularly varying at infinity with index $-p q$, we conclude that
\begin{equation}\label{eq:O(mu)}
     \cO(\Fbar(t)^q) = \cO(\Fbar(\lambda t)^q) \quad \text{for all }\lambda > 0.
\end{equation}
Furthermore, by Potter's theorem (e.g.~\cite[Theorem 1.5.6]{Bingham1987}) for every $\delta > 0$ there exists $c > 0$ such that for all suficiently large $t$ we have $\Fbar(t)^q \geq c t^{-pq - \delta}$.
This implies that for every $\delta > 0$ we have
\begin{equation}\label{eq:exp-in-o}
     \e^{-\delta t} \in o(\Fbar(t)^{q}).
\end{equation}
Finally, by Theorem \ref{thm:spinal-decomp} (iii) we have, for every $\theta > 0$ and $n \in \N$,
\begin{equation}\label{eq:exp-spine}
     \bbE^\cQ[\e^{-\theta V(w_{n})}] = \E[\e^{-\theta S_{n}}] = \E[\e^{-\theta S_1}]^n = m(1+\theta)^n.
\end{equation}
Using \eqref{eq:heavy-tailed}, \eqref{eq:positive-mean} and \eqref{eq:Lp-cond} combined with the convexity of $m$ we see that $m(1+\theta) < 1$ for all $\theta \in (0,\gamma -1)$.

Now we prove the proposition.
Let $t, M \in (0, \infty)$ and $n \in \N$. 
Using Theorem \ref{thm:spinal-decomp} (ii) we have
\begin{equation}\label{eq:Zhat2spine}
     \Zbar_n(t) = W_n \sum_{|u|=n} \frac{\e^{-V_u}}{W_n} \1\{V_u > t\} = W_n \cQ(V(w_n) > t\,|\,\F_n).
\end{equation}
By switching to the measure $\cQ$ and using properties of conditional expectations we find
\begin{align*}
     \bbE[\Zbar_n(t)^{1 + \eta}\1\{W_n \leq M\}] &= \bbE^{\cQ}\Big[ \Zbar_n(t)^{\eta} \frac{\Zbar_n(t)}{W_n}\1\{W_n \leq M\} \Big] \\
     &= \bbE^{\cQ}[\Zbar_n(t)^{\eta} \cQ(V(w_n) > t\,|\,\F_n) \1\{W_n \leq M\}] \\
     &= \bbE^{\cQ}[\Zbar_n(t)^{\eta} \1\{V(w_n) > t,\,W_n \leq M\} ].
\end{align*}
Now we condition on the spine and use Jensen's inequality for conditional expectations to get
\begin{align}
     \begin{split}\label{eq:Zhat-gamma-moment}
          \bbE[\Zbar_n(t)^{1 + \eta} \1\{W_n \leq M\}] &= \bbE^{\cQ}[\bbE^{\cQ}[\Zbar_n(t)^{\eta}\1\{W_n \leq M\}\,|\,\G_n] \1\{V(w_n) > t\} ] \\
          &\leq \bbE^{\cQ}[\bbE^{\cQ}[\Zbar_n(t)\1\{W_n \leq M\}\,|\,\G_n]^\eta \1\{V(w_n) > t\} ].
     \end{split}
\end{align}
Note that the inner conditional expectation exists since $0 \leq \Zbar_n(t) \leq W_n \leq M$.
Decomposition along the spine yields
\begin{align*}
     \Zbar_n(t) = \e^{-V(w_n)}\1\{V(w_n) > t\} + \sum_{j=1}^{n} \sum_{u \in \Omega(w_{j})} \e^{-V_u}[\Zbar_{n-j}(t-V_u)]_u
\end{align*}
where $\Zbar_0(t) = \1\{t < 0\}$.
Therefore, we get
\begin{align*}
     \bbE^{\cQ}[\Zbar_n(t)\1\{W_n \leq M\}\,|\,\G_n] &= \e^{-V(w_n)}\1\{V(w_n) > t\}\bbE^\cQ[ \1\{W_n \leq M\} \,|\,\G_n] \\
     &+ \sum_{j=1}^{n} \sum_{u \in \Omega(w_{j})} \e^{-V_u}\bbE^\cQ\big[[\Zbar_{n-j}(t-V_u)]_u \1\{W_n \leq M\}\,|\, \G_n \big].
\end{align*}
Bounding $\1\{W_n \leq M\} \leq 1$ and applying the many-to-one formula results in
\begin{equation*}
     \bbE^{\cQ}[\Zbar_n(t)\1\{W_n \leq M\}\,|\,\G_n] \leq \e^{-V(w_n)}\1\{V(w_n) > t\} + \sum_{j=1}^{n} \sum_{u \in \Omega(w_{j})} \e^{-V_u}\Fbar_{n-j}(t - V_u - (n-j)c)
\end{equation*}
where $\Fbar_j(t) \coloneqq \Prob(S_j - jc > t)$, $j > 0$, $\Fbar_0(t) \coloneqq \1\{t < 0\}$ should be recalled.
Now we plug this back in \eqref{eq:Zhat-gamma-moment}, use that $x \mapsto x^\eta$ is subadditive and take $M \uparrow \infty$ to get
\begin{align}\label{eq:spine-decomp}
     \begin{split}
          \bbE[\Zbar_n(nc+t_n)^{1 + \eta}] &\leq \bbE^\cQ \bigg[\e^{-\eta V(w_n)}\1\{V(w_n) > nc+t_n\} \\
          &+ \sum_{j=1}^{n} \bbE^\cQ\bigg[\1\{V(w_n) > nc+t_n\} \bigg(\sum_{u \in \Omega(w_{j})} \e^{-V_u}\Fbar_{n-j}\big(t_n - (V_u-jc) \big) \bigg)^\eta \bigg].
     \end{split}
\end{align}
The first summand on the right-hand side is bounded by $\e^{-\eta t_n}$ which is in $o((n \Fbar(t_n))^{1 + \eta})$ by \eqref{eq:exp-in-o}.
The rest of the proof is concerned with estimating the second summand.
We proceed in several steps.
For brevity, we write $t_n' \coloneqq nc + t_n$ and
\begin{align*}
     \Wbar_n(t) &\coloneqq \sum_{u \in \Omega(w_n)}\e^{-\Delta V_u} \1\{\Delta V_u > t\},\quad t \in \R \\
     \Wbar_n &\coloneqq \sum_{u \in \Omega(w_n)}\e^{-\Delta V_u},\quad n \in \N
\end{align*}
where $\Delta V_u = V_u - V(w_n)$ is the displacement of $u \in \Omega(w_n)$ with respect to its parent.

\noindent
\textbf{Step 1:} \emph{The summand in \eqref{eq:spine-decomp} with $j=n$.}

The summand on the right-hand side of \eqref{eq:spine-decomp} with $j=n$ is given by
\begin{align}\label{eq:summand-j=n}
     \begin{split}
          &\bbE^\cQ \bigg[\1\{V(w_n) > t_n'\} \bigg( \sum_{u \in \Omega(w_n)} \e^{-V_u} \1\{V_u > t_n'\}\bigg)^\eta \bigg] \\
          =\quad &\bbE^\cQ \big[\1\{V(w_n) > t_n'\} \e^{-\eta V(w_{n-1})} \Wbar_n(t_n' - V(w_{n-1}))^\eta \big] \\
          =\quad &\bbE^\cQ \big[\e^{-\eta V(w_{n-1})}\bbE^\cQ[\1\{V(w_n) > t_n'\} \Wbar_n(t_n' - V(w_{n-1}))^\eta \,|\,\G_{n-1}] \big] \\
          =\quad &\bbE^\cQ[\e^{-\eta V(w_{n-1})} \phi(t_n' - V(w_{n-1})) ]
     \end{split}
\end{align}
where
\begin{equation}\label{eq:stochdom-phi}
     \phi(t) \coloneqq \bbE^\cQ[ \1\{V(w_1) > t\} \Wbar_1(t)^\eta],\quad t \in \R.
\end{equation}
We have, for all $t \in \R$,
\begin{equation*}
     \phi(t) \leq \bbE^\cQ[\Wbar_1^\eta] \leq \bbE^\cQ[W_1^\eta] = \bbE[W_1^{{1 + \eta}}] < \infty.
\end{equation*}
Furthermore, using \eqref{eq:Zhat2spine} with $n=1$ we see that
\begin{align*}
     \phi(t) = \bbE^\cQ\Big[\frac{\Zbar_1(t)}{W_1} \Wbar_1(t)^\eta  \Big] \leq \bbE^\cQ\Big[ \frac{\Zbar_1(t)^{1+\eta}}{W_1} \Big] = \bbE[\Zbar_1(t)^{1 + \eta}].
\end{align*}
By \eqref{eq:Zhat-moment-tail} this is in $\cO(\Fbar(t)^{1 + \eta})$ as $t \to \infty$.
Thus, for the right-hand side of \eqref{eq:summand-j=n} we infer, using that $\phi$ is decreasing,
\begin{align*}
     &\bbE^\cQ[\e^{-\eta V(w_{n-1})} \phi(t_n' - V(w_{n-1})) ] \\
     \leq \quad &\bbE^\cQ[\e^{-\eta V(w_{n-1})} \phi(t_n - V(w_{n-1})) ] \\
     \leq \quad &\bbE^\cQ[\e^{-\eta V(w_{n-1})} \1\{V(w_{n-1}) > t_n/2\} ]\bbE[W_1^{{1 + \eta}}] \\
     &+ \phi(t_n/2)\bbE^\cQ[\e^{-\eta V(w_{n-1})} \1\{V(w_{n-1}) \leq t_n/2\} ] \\
     \leq \quad &\e^{-\eta t_n/2} \bbE[W_1^{{1 + \eta}}] + \phi(t_n/2)\bbE^\cQ[\e^{-\eta V(w_{n-1})}]
\end{align*}
The first summand on the right-hand side is in $o(\Fbar(t_n)^{1 + \eta})$ by \eqref{eq:exp-in-o}.
Regarding the second summand, we use \eqref{eq:exp-spine} and \eqref{eq:O(mu)} to get
\begin{align*}
     \phi(t_n/2)\bbE^\cQ[\e^{-\eta V(w_{n-1})}] &= \phi(t_n/2)m(1+\eta)^n \\
     &\leq \phi(t_n/2) \in \cO(\Fbar(t_n/2)^{1 + \eta}) = \cO(\Fbar(t_n)^{1 + \eta})
\end{align*}
Combined we have shown that the $j=n$ summand on the right-hand side of \eqref{eq:spine-decomp} is in $\cO\big( \Fbar(t_n)^{1 + \eta} \big)$.

\noindent
\textbf{Step 2:} \emph{The summands in \eqref{eq:spine-decomp} with $j \in \{1,\ldots,n-1\}$.}

We claim that there exists $C \in (0,\infty)$ such that for all sufficiently large $n \in \N$ and $j < n$ we have
\begin{equation}\label{eq:summand-j<n-claim}
     \bbE^\cQ\bigg[\1\{V(w_n) > t_n'\} \bigg(\sum_{u \in \Omega(w_{j})} \e^{-V_u}\Fbar_{n-j}\big(t_n - (V_u-jc)\big) \bigg)^\eta \bigg] \leq C (n \Fbar(t_n))^{1 + \eta} m(1+\eta/2)^j.
\end{equation}
Note that, using \eqref{eq:spine-decomp} and $m(1+\eta/2)<1$, this claim combined with step 1 completes the proof. 

Let $\eps > 0$ such that $a(1-\eps) > \sqrt{p-2}$ and $\eps/2 \leq 1 - \eps$.
By Theorem \ref{thm:Nagaev} combined with \eqref{eq:subexponential} there exists $n_0 \in \N$ such that for all $n \geq n_0$, $j \leq n$ and $t \geq (1-\eps) a \sigma \sqrt{n \log n}$ we have
\begin{equation}\label{eq:Fbar_bound}
\Fbar_j(t) \leq 2 j \Fbar(t).
\end{equation}
Indeed, first we take $n_1$ such that for all $n \geq n_1$ and $t \geq (1-\eps) a \sigma \sqrt{n \log n}$ we have $\Fbar_n(t) \leq 2 n \Fbar(t)$.
Then, by \eqref{eq:subexponential}, there exists $t_0$ such that $\Fbar_j(t) \leq 2 j \Fbar(t)$ for all $j \leq n_1$ and $t \geq t_0$.
Now take $n_0 \geq n_1$ such that $(1-\eps) a \sigma \sqrt{n_0 \log n_0} \geq t_0$ and \eqref{eq:Fbar_bound} follows.
Moreover, by \eqref{eq:O(mu)} we may choose $n_0$ large enough to ensure the validity of
\begin{equation}\label{eq:muhat-extra-arg}
     \Fbar((1-\eps)t_n) \leq \Fbar(\eps t_n/2) \leq K \Fbar(t_n)
\end{equation}
for some constant $K \in (0,\infty)$ and all $n \geq n_0$ (the first inequality is due to $\eps/2 \leq 1 - \eps$).
For the rest of the proof we assume that $n \geq n_0$ holds.

We proceed with the estimation of the left-hand side of \eqref{eq:summand-j<n-claim}.
Let $j < n$.
Using conditioning on $\G_{j}$, Theorem \ref{thm:spinal-decomp} (iii) and the fact that $\Fbar_j$ is decreasing we find
\begin{align}\label{eq:summand-j<n}
     \begin{split}
     &\bbE^\cQ\bigg[\1\{V(w_n) > t_n'\} \bigg(\sum_{u \in \Omega(w_{j})} \e^{-V_u}\Fbar_{n-j}\big(t_n - (V_u-jc)\big) \bigg)^\eta \bigg] \\
     = \quad &\bbE^\cQ\bigg[\bigg(\sum_{u \in \Omega(w_{j})} \e^{-V_u}\Fbar_{n-j}\big(t_n - (V_u-jc)\big) \bigg)^\eta \cQ(V(w_n) > t_n'\,|\,\G_j) \bigg] \\
     = \quad &\bbE^\cQ\bigg[\e^{-\eta V(w_{j-1})} \bigg(\sum_{u \in \Omega(w_{j})} \e^{-\Delta V_u}\Fbar_{n-j}\big(t_n - (V_u-jc)\big) \bigg)^\eta \Fbar_{n-j}\big(t_n - (V(w_j)-jc)\big) \bigg] \\
     \leq \quad &\bbE^\cQ\bigg[\e^{-\eta V(w_{j-1})} \bigg(\sum_{u \in \Omega(w_{j})} \e^{-\Delta V_u}\Fbar_{n-j}(t_n - V_u) \bigg)^\eta \Fbar_{n-j}(t_n - V(w_j)) \bigg]
     \end{split}
\end{align}
First we estimate the inner term in brackets.
We split the sum into particles $u \in \Omega(w_j)$ with $V_u > \eps t_n$ or $V_u \geq \eps t_n$.
For $V_u > \eps t_n$ we estimate $\Fbar_{n-j}(t_n - V_u) \leq 1$.
For $V_u \leq \eps t_n$ we use again that $\Fbar_{n-j}$ is decreasing and get by \eqref{eq:Fbar_bound} and \eqref{eq:muhat-extra-arg}
\begin{equation*}
     \Fbar_{n-j}(t_n - V_u) \leq \Fbar_{n-j}((1-\eps)t_n) \leq 2(n-j)\Fbar((1-\eps) t_n) \leq 2Kn \Fbar(t_n).
\end{equation*}
Therefore, we have
\begin{align*}
     &\sum_{u \in \Omega(w_{j})} \e^{-\Delta V_u}\Fbar_{n-j}(t_n - V_u) \\
     \leq \quad &\sum_{\substack{u \in \Omega(w_{j}) \\ V_u > \eps t_n}} \e^{-\Delta V_u} + \sum_{\substack{u \in \Omega(w_{j}) \\ V_u \leq \eps t_n}} \e^{-\Delta V_u}2K n\Fbar(t_n) \\
     \leq \quad &\Wbar_j(\eps t_n - V(w_{j-1})) + 2Kn\Fbar(t_n)\Wbar_j.
\end{align*}
Since $x \mapsto x^\eta$ is increasing and subadditive, we get from \eqref{eq:summand-j<n}
\begin{align}\label{eq:summand-j<n2}
     \begin{split}
          &\bbE^\cQ\bigg[\1\{V(w_n) > t_n'\} \bigg(\sum_{u \in \Omega(w_{j})} \e^{-V_u}\Fbar_{n-j}\big(t_n - (V_u-jc)\big) \bigg)^\eta \bigg] \\
          \leq \quad &\bbE^\cQ\big[\e^{-\eta V(w_{j-1})} \Wbar_j(\eps t_n - V(w_{j-1}))^\eta \Fbar_{n-j}(t_n - V(w_j)) \big] \\
          &+ (2Kn\Fbar(t_n))^\eta \bbE^\cQ\big[\e^{-\eta V(w_{j-1})} \Wbar_j^\eta \Fbar_{n-j}(t_n - V(w_j)) \big].
     \end{split}
\end{align}
It suffices to bound each summand of the right-hand side of this inequality by the right-hand side of \eqref{eq:summand-j<n-claim}.
Our strategy to do so is to split the expectation according to the values of $V(w_{j-1})$ and $V(w_j)$.
To organize this we define the events
\begin{align*}
     A_{j,n} \coloneqq \{V(w_{j-1}) \leq \eps t_n/2 \},\quad B_{j,n} \coloneqq \{V(w_j) \leq \eps t_n\},\quad j \leq n.
\end{align*}

\noindent
\textbf{Step 2.1}: \emph{Second summand on the right-hand side of \eqref{eq:summand-j<n2}.}

To control the second summand on the right-hand side of \eqref{eq:summand-j<n2} it suffices to show
\begin{equation*}
     \bbE^\cQ\big[\e^{-\eta V(w_{j-1})} \Wbar_j^\eta \Fbar_{n-j}(t_n - V(w_j)) \big] \leq C n \Fbar(t_n) m(1+\eta/2)^j
\end{equation*}
for some $C \in (0,\infty)$, all sufficiently large $n$ and all $j < n$.
On the event $B_{j,n}$ we have, by \eqref{eq:Fbar_bound} and \eqref{eq:muhat-extra-arg},
\begin{align}\label{eq:Fbar-small-V}
     \begin{split}
          \Fbar_{n-j}(t_n - V(w_j)) &\leq \Fbar_{n-j}((1-\eps) t_n) \\
          &\leq 2 (n-j) \Fbar((1-\eps)t_n) \leq 2 K n \Fbar(t_n).
     \end{split}
\end{align}
Therefore, we get using \eqref{eq:exp-spine}
\begin{align*}
     &\bbE^\cQ\big[\e^{-\eta V(w_{j-1})} \Wbar_j^\eta \Fbar_{n-j}(t_n - V(w_j)) \1_{B_{j,n}}\big] \\
     \leq \quad &2 K n \Fbar(t_n) \bbE^\cQ\big[\e^{-\eta V(w_{j-1})} \Wbar_j^\eta \big] \\
     = \quad &2 K n \Fbar(t_n) \bbE^\cQ[\e^{-\eta V(w_{j-1})}] \bbE^\cQ[\Wbar_j^\eta] \\
     \leq \quad &2 K n \Fbar(t_n) m(1 + \eta)^{j-1} \bbE[W_1^{1 + \eta}].
\end{align*}
The right-hand side has the desired form.
Furthermore, on $B_{j,n}^\comp$ we estimate $\Fbar_{n-j}(t_n - V(w_j)) \leq 1$ and split into to $A_{j,n}$ and $A_{j,n}^\comp$.
On $A_{j,n}^\comp$ we estimate $\e^{-\eta V(w_{j-1})} \leq \e^{-\eta \eps t_n/4} \e^{-\eta V(w_{j-1})/2}$ and obtain
\begin{align*}
     &\bbE^\cQ[\e^{-\eta V(w_{j-1})} \Wbar_j^\eta \Fbar_{n-j}(t_n - V(w_j)) \1_{B_{j,n}^\comp \cap A_{j,n}^\comp}] \\
     \leq \quad &\e^{-\eta \eps t_n/4} \bbE^\cQ[\e^{-\eta V(w_{j-1})/2} \Wbar_j^\eta] \\
     \leq \quad &\e^{-\eta \eps t_n/4} m(1 + \eta/2)^{j-1} \bbE[W_1^{1 + \eta}].
\end{align*}
By \eqref{eq:exp-in-o}, this is also of the desired form.
On $B_{j,n}^\comp \cap A_{j,n}$ we necessarily have $\Delta V(w_j) > \eps t_n/2$, so that
\begin{align*}
     &\bbE^\cQ[\e^{-\eta V(w_{j-1})} \Wbar_j^\eta \Fbar_{n-j}(t_n - V(w_j)) \1_{B_{j,n}^\comp \cap A_{j,n}}] \\
     \leq \quad &\bbE^\cQ[\e^{-\eta V(w_{j-1})} \Wbar_j^\eta \1\{\Delta V(w_j) > \eps t_n/2\}] \\
     = \quad &\bbE^\cQ[\e^{-\eta V(w_{j-1})}] \bbE^\cQ[\Wbar_1^\eta \1\{V(w_1) > \eps t_n/2\}] \\
     = \quad &m(1 + \eta)^{j-1} \bbE^\cQ \Big[ \Wbar_1^{\eta} \frac{\Zbar_1(\eps t_n/2)}{W_1} \Big] \\
     \leq \quad &m(1 + \eta)^{j-1} \bbE[W_1^\eta \Zbar_1(\eps t_n/2) ]
\end{align*}
where we have used \eqref{eq:exp-spine} and \eqref{eq:Zhat2spine} with $n=1$ in the penultimate step.
By Hölder's inequality ($\mathsf{p} = (1 + \eta)/\eta, \mathsf{q} = 1 + \eta$), we infer
\begin{equation*}
     \bbE[W_1^\eta \Zbar_1(\eps t_n/2) ] \leq \bbE[W_1^{1 + \eta}]^\frac{\eta}{1 + \eta} \bbE[\Zbar_1(\eps t_n/2)^{1 + \eta}]^\frac{1}{1+\eta},
\end{equation*}
which by \eqref{eq:Zhat-moment-tail} is in $\cO(\Fbar(\eps t_n/2)) = \cO(\Fbar(t_n))$.
This finishes the second summand on the right-hand side of \eqref{eq:summand-j<n2}.
Now we turn to the first summand and proceed similarly.

\noindent
\textbf{Step 2.2}: \emph{First summand on the right-hand side of \eqref{eq:summand-j<n2}.}

Again, we use the events $A_{j,n}$ and $B_{j,n}$ to split the expectation into several parts according to the values of $V(w_{j-1})$ and $V(w_{j})$.
As before we get
\begin{align*}
     &\bbE^\cQ\big[\e^{-\eta V(w_{j-1})} \Wbar_j(\eps t_n - V(w_{j-1}))^\eta \Fbar_{n-j}(t_n - V(w_j)) \1_{A_{j,n}^\comp}\big] \\
     \leq \quad &\e^{-\eps \eta t_n/4} \bbE^\cQ\big[\e^{-\eta V(w_{j-1})/2} \Wbar_j^\eta] \\
     \leq \quad &\e^{-\eps \eta t_n/4} m(1+\eta/2)^{j-1} \bbE[W_1^{1 + \eta}].
\end{align*}
Further, by \eqref{eq:Fbar-small-V} we see that
\begin{align*}
     &\bbE^\cQ\big[\e^{-\eta V(w_{j-1})} \Wbar_j(\eps t_n - V(w_{j-1}))^\eta \Fbar_{n-j}(t_n - V(w_j)) \1_{A_{j,n} \cap B_{j,n}}\big] \\
     \leq \quad &2 K n \Fbar(t_n) \bbE^\cQ\big[\e^{-\eta V(w_{j-1})} \Wbar_j(\eps t_n/2))^\eta \big] \\
     = \quad &2 K n \Fbar(t_n) m(1 + \eta)^{j-1} \bbE^\cQ[\Wbar_1(\eps t_n/2)^\eta].
\end{align*}
By Hölder's inequality (with $\mathsf{p} = 1/(1 - \eta) = \gamma, \mathsf{q} = 1/\eta$, recall that $\eta = 1 - \gamma^{-1}$) we have for all $t > 0$
\begin{align*}
     \bbE^\cQ[\Wbar_1(t)^\eta] &= \bbE[W_1 \Wbar_1(t)^\eta] \leq \bbE[W_1 \Zbar_1(t)^\eta] \\
     &\leq \bbE[W_1^{\gamma}]^{1/\gamma}\bbE[\Zbar_1(t)]^\eta = \bbE[W_1^{\gamma}]^{1/\gamma} \Fbar(t)^\eta.
\end{align*}
Thus, $\bbE^\cQ[\Wbar_1(\eps t_n/2)^\eta] \in \cO(\Fbar(\eps t_n/2)^\eta) = \cO(\Fbar(t_n)^\eta)$.
Finally, we have
\begin{align*}
     &\bbE^\cQ\big[\e^{-\eta V(w_{j-1})} \Wbar_j(\eps t_n - V(w_{j-1}))^\eta \Fbar_{n-j}(t_n - V(w_j)) \1_{A_{j,n} \cap B_{j,n}^\comp}\big] \\
     \leq \quad &\bbE^\cQ[ \e^{-\eta V(w_{j-1})} \Wbar_j(\eps t_n/2)^\eta \1\{\Delta V(w_j) > \eps t_n/2\} ] \\
     = \quad &\bbE^\cQ[ \e^{-\eta V(w_{j-1})}] \bbE^\cQ[\Wbar_j(\eps t_n/2)^\eta \1\{\Delta V(w_j) > \eps t_n/2\}] \\
     = \quad &m(1 + \eta)^{j-1} \phi(\eps t_n/2)
\end{align*}
with $\phi$ from \eqref{eq:stochdom-phi}.
We have shown earlier that $\phi(t) \in \cO(\Fbar(t)^{1 + \eta})$, thus combined with \eqref{eq:O(mu)} the claimed bound follows.
This concludes the proof.
\end{proof}

\appendix
\section{A Marcinkiewicz-Zygmund-type weak law of large numbers}\label{sec:wlln-proof}
Recall that we write $X \preceq Y$ if for all sufficiently large $t \geq 0$ we have $\bbP(X > t) \leq \bbP(Y > t)$.
To prove the following Marcinkiewicz-Zygmund-type weak law of large numbers one can adapt \cite[Theorem 6.17]{Kallenberg2021} straight-forwardly.
\begin{Thm}\label{thm:wlln}
     Let $(X_{j})_{j \in \N}$ be a sequence of independent random variables.
     Suppose further that $\bbE[X_{j}] = 0$ and that there exists a random variable $X \geq 0$ such that $|X_{j}| \preceq X$ for all $j \in \N$, and that we have $\bbP(X > t) \in o(t^{-p})$ as $t \to \infty$ for some $p \in (1,2)$.
     Then we have $n^{-1/p} \sum_{j=1}^n X_j \to 0$ in probability as $n \to \infty$.
\end{Thm}
\begin{proof}[Proof of Theorem \ref{thm:wlln-w/growth}.]
     We modify the the proof of \cite[Proposition 4.1]{Nerman1981}.
     Assume without loss of generality that $\bbE[X_{kj}] = 0$ for all $k,j$ and that the whole family $(X_{kj})_{j \leq n_k,\,k \in \N}$ is independent.
     By Theorem \ref{thm:wlln} we have
     \begin{equation*}
          \frac{\sum_{i=1}^k \sum_{j=1}^{n_i} X_{ij}}{\Big(\sum_{i = 1}^k n_i\Big)^{1/p}} \pconv{k}{\infty} 0
     \end{equation*}
     Then we find
     \begin{align*}
          \frac{1}{n_k^{q/p}} \sum_{j=1}^{n_k} X_{kj} = \frac{\sum_{i=1}^k \sum_{j=1}^{n_i} X_{ij}}{\Big(\sum_{i = 1}^k n_i\Big)^{1/p}} \bigg(\frac{\sum_{i = 1}^k n_i}{n_k^q} \bigg)^{1/p} - \frac{\sum_{i=1}^{k-1} \sum_{j=1}^{n_i} X_{ij}}{\Big(\sum_{i = 1}^{k-1} n_i\Big)^{1/p}} \bigg(\frac{\sum_{i = 1}^{k-1} n_i}{n_k^q} \bigg)^{1/p}.
     \end{align*}
     Since the right-hand side converges to zero in probability the claim is proved.
\end{proof}

\section{Details on Example \ref{exa:example} (b)}\label{sec:example}
Here we prove the remaining claims of Example \ref{exa:example} (b).
Suppose that we have $\Prob(1-f \leq t) = t^{p+1} \ell(t)$ for all $t \leq \eps \in (0,1)$, where $\ell$ is slowly varying at zero.

\emph{Claim 1:} $m(1) < \infty$.
We have
\begin{equation*}
\frac{m(1)}{b} = \E[(1-f)^{-1}] = \E[(1-f)^{-1}\1\{1-f > \eps\}] + \E[(1-f)^{-1}\1\{1-f \leq \eps\}].
\end{equation*}
The first summand on the right-hand side is bounded by $1/\eps$.
For the second one we use the identity $t^{-1} = \int_t^\infty x^{-2} \d x$ and find, using Fubini's theorem,
\begin{align*}
\E[(1-f)^{-1}\1\{1-f \leq \eps\}] &= \E\bigg[\int_{1-f}^\infty x^{-2}\1\{1-f \leq \eps\} \d x\bigg] \\
&= \int_0^\infty x^{-2}\Prob(1-f \leq \eps \wedge x) \d x \\
&= \int_0^\eps x^{-2} \Prob(1-f \leq x) \d x + \int_\eps^\infty x^{-2} \Prob(1-f \leq \eps) \d x \\
&= \int_0^\eps x^{p-1}\ell(x) \d x + \Prob(1-f \leq \eps) \eps^{-1}.
\end{align*}
Now note that $\tilde{\ell}(x) \coloneqq \ell(1/x)$ is slowly varying at $\infty$, and the first summand transforms to
\begin{equation*}
\int_0^\eps x^{p-1}\ell(x) \d x = \int_{\eps^{-1}}^\infty y^{-p-1}\tilde{\ell}(y) \d y
\end{equation*}
which is finite by \cite[Proposition 1.5.10]{Bingham1987} since $p > 0$.
Therefore, $m(1) < \infty$ as claimed.

\emph{Claim 2:} $t \mapsto \E[\e^{-t(1-f)}]$ is regularly varying at $\infty$ with index $-(p+1)$.
We have
\begin{equation*}
\E[\e^{-t(1-f)}] = \E[\e^{-t(1-f)}\1\{1-f > \eps\}] + \E[\e^{-t(1-f)} \1\{1-f \leq \eps\}].
\end{equation*}
The first summand on the right-hand side is bounded by $\e^{-\eps t}$ and hence of lower order once we show that there is a regularly varying part.
for the second summand we use the identity $\e^{-t} = \int_t^\infty \e^{-x} \d x$ and find, using Fubini's theorem,
\begin{align*}
\E[\e^{-t(1-f)} \1\{1-f \leq \eps\}] &= \E\bigg[ \int_{t(1-f)}^\infty \e^{-x}\1\{1-f \leq \eps\} \d x \bigg] \\
&= \int_0^\infty \e^{-x} \Prob(1-f \leq \eps \wedge x/t) \d x \\
&= \int_0^{\eps t} \e^{-x} \Prob(1-f \leq x/t) \d x + \int_{\eps t}^\infty \e^{-x} \Prob(1-f \leq \eps) \d x.
\end{align*}
Once again, the second summand on the right-hand side is bounded by $\e^{-\eps t}$.
The first one is further evaluated as
\begin{equation*}
\int_0^{\eps t} \e^{-x} \Prob(1-f \leq x/t) \d x = t^{-(p+1)} \int_0^{\eps t} \e^{-x} x^{p+1} \ell(x/t) \d x.
\end{equation*}
We claim that, as $t \to \infty$,
\begin{equation*}
\int_0^{\eps t} \e^{-x} x^{p+1} \ell(x/t) \d x \sim \ell(1/t) \Gamma(p+2)
\end{equation*}
where $\Gamma$ denotes the Gamma-function.
Note that for each fixed $x > 0$ we have $\ell(x/t)/\ell(1/t) \to 1$ as $t \to \infty$, thus this claim follows from the dominated convergence theorem once we find a majorant.
Define the function $\tilde{\ell}(x) = \ell(1/x)$ for $x \geq 1/\eps$ and $\tilde{\ell}(x) = \ell(\eps)$ for $x < 1/\eps$.
Then $\tilde{\ell}$ is slowly varying at $\infty$ and bounded away from $0$ and $\infty$ on all compact intervals in $(0,\infty)$ (since otherwise we run into trouble with $\Prob(1-f \leq t) = t^p \ell(t)$).
By Potter's theorem (e.g.~\cite[Theorem 1.5.6]{Bingham1987}) there exists $C \in (0,\infty)$ such that
\begin{equation*}
\frac{\tilde{\ell}(t/x)}{\tilde{\ell}(t)} \leq C x \quad \text{for all }t/x, t > 0.
\end{equation*}
Now we see that
\begin{equation*}
\frac1{\ell(1/t)} \int_0^{\eps t} \e^{-x} x^{p+1} \ell(x/t) \d x = \int_0^{\eps t} \e^{-x} x^{p+1} \frac{\tilde{\ell}(t/x)}{\tilde{\ell}(t)} \d x \leq C\int_0^\infty \e^{-x} x^{p+2} \d x < \infty.
\end{equation*}
This gives the desired majorant.
We have thus shown that, as $t \to \infty$,
\begin{equation*}
\E[\e^{-t(1-f)}] \sim t^{-(p+1)}\ell(1/t) \Gamma(p+2).
\end{equation*}
Therefore, the claim is proved.

\section*{Acknowledgement}
The author wishes to express sincere gratitude to Alicja Ko\l odziejska and Matthias Meiners for carefully reading an early version of this manuscript.
This work was financially supported by DFG grant ME 3625/5-1.

\bibliographystyle{alpha}
\bibliography{large-deviations}

\end{document}